\documentclass[11pt]{amsart}

\usepackage[T1]{fontenc}
\usepackage[utf8]{inputenc}
\usepackage{amsmath,amssymb,amsthm,mathtools,mathrsfs}
\usepackage{enumitem}
\usepackage{booktabs,array,tabularx}
\usepackage{microtype}
\usepackage[margin=1in]{geometry}
\usepackage{xcolor}
\usepackage{hyperref}
\hypersetup{
  unicode=true,
  colorlinks=true,
  linkcolor=blue!48!black,
  citecolor=green!32!black,
  urlcolor=blue!55!black,
  pdftitle={Rank Amplification for Shifted Equal Values of Euler's Totient Function},
  pdfauthor={Eric Li},
  pdfsubject={Euler function, shifted equal values, rank amplification},
  pdfkeywords={Euler function, shifted equal values, rank amplification, chains, smooth values, friable values, formal verification}
}

\makeatletter
\newcommand{\affiliation}[1]{\gdef\@affiliation{#1}}
\let\@affiliation\@empty
\def\@setauthors{%
  \begingroup
  \def\thanks{\protect\thanks@warning}%
  \trivlist
  \centering\footnotesize \@topsep30\p@\relax
  \advance\@topsep by -\baselineskip
  \item\relax
  \author@andify\authors
  \def\\{\protect\linebreak}%
  \MakeUppercase{\authors}\par
  \ifx\@empty\@affiliation\else
    \vskip 8\p@
    {\normalfont\normalsize\@affiliation\par}%
  \fi
  \ifx\@empty\contribs
  \else
    ,\penalty-3 \space \@setcontribs
    \@closetoccontribs
  \fi
  \endtrivlist
  \endgroup
}
\makeatother

\newtheorem{theorem}{Theorem}[section]
\newtheorem{proposition}[theorem]{Proposition}
\newtheorem{lemma}[theorem]{Lemma}
\newtheorem{corollary}[theorem]{Corollary}
\newtheorem{definition}[theorem]{Definition}
\theoremstyle{remark}
\newtheorem{remark}[theorem]{Remark}

\numberwithin{equation}{section}

\newcommand{\rad}{\operatorname{rad}}
\newcommand{\sA}{\mathscr A}

\title[Rank amplification for Euler's totient function]
{Rank Amplification for Shifted Equal Values of Euler's Totient Function}

\author{Eric Li}
\affiliation{Trinity College, University of Cambridge}
\thanks{Email addresses: \href{mailto:el593@cam.ac.uk}{el593@cam.ac.uk}, \href{mailto:contact@ericli.com}{contact@ericli.com}.}
\thanks{All results of this paper have been formally verified in Lean; see Section~\ref{sec:formalisation} and~\cite{LeanFormalization}.}
\date{June 22, 2026}

\begin{document}

\begin{abstract}
Let
\[
 S_h^\varphi(x)=\#\{n\le x:\varphi(n)=\varphi(n+h)\}.
\]
For the unit shift we prove
\[
 S_1^\varphi(x)
 \ll x\exp\!\left\{-(\tfrac12-o(1))
        \sqrt{\log x\,\log_2 x}\right\}.
\]
More generally, put
\[
 \sA=\log_3 x+\log_4 x-\log 2,
 \qquad G=\sqrt{\log x\,\sA},
 \qquad V=\frac{\log x}{G}.
\]
For every fixed integer $J\ge1$, uniformly for
$1\le h\le\exp\{G/\sqrt J\}$, we obtain
\[
 S_h^\varphi(x)=D_{h,>e^{\sqrt JG}}^\varphi(x)
 +O_J\!\left(x\exp\{-\sqrt JG+o_J(V)\}\right).
\]
Here $D_{h,>Y}^\varphi(x)$ is the above-cutoff part of the classical
Graham--Holt--Pomerance same-support family; it is empty for odd $h$.
A moving choice $J\asymp\log_2 x/\log_3 x$ gives the first displayed
unit-shift estimate and an analogous decomposition for a uniform range of
shifts.  The proof combines the smooth-totient theorem of
Banks--Friedlander--Pomerance--Shparlinski with labelled supplier systems,
a shifted divisor convolution, and an injective encoding of large supplier
products into weighted friable tuples.  This gives a quantitative separation
of the GHP diagonal from the off-diagonal solutions at both fixed and moving
rank.  The results of this paper have been formally verified in Lean.
\end{abstract}

\maketitle

\section{Introduction}

Let $P^+(m)$ denote the largest prime factor of $m$, with $P^+(1)=1$, and
let $\log_j x$ denote the $j$-fold iterated logarithm.  For a positive
integer $h$, put
\[
 S_h^\varphi(x)=\#\{n\le x:\varphi(n)=\varphi(n+h)\},
 \qquad S(x)=S_1^\varphi(x).
\]
The unit-shift equation $\varphi(n)=\varphi(n+1)$ is a classical test
problem for shifted values of Euler's function.  Erd\H{o}s--Pomerance--S\'ark\H{o}zy proved the unit-shift bound
\[
 S_1^\varphi(x)\ll x\exp\{-(\log x)^{1/3}\}
\]
for all sufficiently large $x$ \cite{EPS}.  Graham--Holt--Pomerance
introduced the same-support prime-pair
construction and analysed its role in shifted equal totients
\cite[Theorems~1 and~2]{GHP}.  Pollack--Pomerance--Trevi\~no gave a
uniform decomposition into this family and its complement for
$h\le\exp\{(\log x)^{1/3}\}$ and proved the quantitative bound for the
finite same-support index set used below
\cite[Theorem~3.1 and Lemma~3.2]{PPT}.  Yamada subsequently proved, for
each fixed odd $h$,
\begin{equation}\label{eq:yamada-intro}
 S_h^\varphi(x)\ll_h x\exp\!\left\{-(2^{-1/2}+o(1))
 (\log x\,\log_3 x)^{1/2}\right\}
\end{equation}
\cite[Theorem~1.2 and Corollary~1.3]{Yamada}.  Kinlaw--Kobayashi--Pomerance
study the reciprocal sum and computation for the unit shift \cite{KKP}.
For even shifts, Kim proved that at least one positive even shift has
infinitely many solutions \cite[Theorem~1.1]{Kim}, and Ford obtained
stronger explicit infinitude results \cite[Theorem~1]{Ford}.

We define the diagonal before stating the main theorem.  For $h\ge1$, let
\[
 \mathcal J_h=\{j\ge1:\rad(j)=\rad(j+h)\}.
\]
For $j\in\mathcal J_h$ put
\[
 d_j=(j,j+h),\qquad A_j=\frac{j+h}{d_j},\qquad
 B_j=\frac{j}{d_j},\qquad \Gamma_j=\varphi(j)A_j.
\]
Equality of prime supports gives
$\Gamma_j=\varphi(j+h)B_j$.  For $Y\ge2$, let
$D_{h,>Y}^\varphi(x)$ be the number of \emph{integers} $n\le x$ for
which at least one pair $(j,r)$ satisfies
\begin{equation}\label{eq:intro-diagonal-param}
 n=j(A_jr+1),\qquad n+h=(j+h)(B_jr+1),
\end{equation}
where $j\in\mathcal J_h$, $r\ge1$,
\[
 A_jr+1\ \text{and}\ B_jr+1\ \text{are prime},
 \qquad (A_jr+1,j)=1,
 \qquad (B_jr+1,j+h)=1,
\]
and $P^+(\Gamma_jr)>Y$.  Multiple representations of the same integer
are counted once.  Thus $D_{h,>Y}^\varphi(x)$ is the above-cutoff tail in the cutoff
parameter.  Pollack--Pomerance--Trevi\~no prove
\[
 |\mathcal J_h|\le 3\cdot7^{3+2\omega(h)}
\]
\cite[Lemma~3.2]{PPT}.  That bound is not needed below: the only uniform
inputs about $\mathcal J_h$ used here are the elementary support,
reciprocal-mass, counting, and parity statements of
Lemma~\ref{lem:J-finite}.

The following compact comparison separates shift range from error scale.
The diagonal in the first three rows is the same classical GHP family.
\begin{center}
\small
\renewcommand{\arraystretch}{1.12}
\begin{tabularx}{\textwidth}{@{}>{\raggedright\arraybackslash}p{0.26\textwidth}>{\raggedright\arraybackslash}p{0.28\textwidth}>{\raggedright\arraybackslash}X@{}}
\toprule
Reference & shift range & off-diagonal upper bound \\
\midrule
Graham--Holt--Pomerance \cite{GHP}
 & fixed $h$ & $x\exp\{-(\log x)^{1/3}\}$ \\
Pollack--Pomerance--Trevi\~no \cite[Theorem~3.1]{PPT}
 & $h\le e^{(\log x)^{1/3}}$
 & $x\exp\{-(\log x)^{1/3}\}$ \\
Yamada \cite[Theorem~1.2]{Yamada}
 & fixed $h$; full count for odd $h$
 & $x\exp\{-(2^{-1/2}+o(1))\sqrt{\log x\,\log_3 x}\}$ \\
Theorem~\ref{thm:main}
 & $h\le e^{G/\sqrt J}$, fixed $J$
 & $x\exp\{-\sqrt JG+o_J(V)\}$ \\
Theorem~\ref{thm:scale-jump}
 & $h\le e^T$
 & $x\exp\{-(\tfrac12-o(1))\sqrt{\log x\,\log_2 x}\}$ \\
\bottomrule
\end{tabularx}
\end{center}
For fixed odd $h$, Theorem~\ref{thm:main} with $J=1$ improves the leading
coefficient in the fixed-odd-shift upper bound at the
$\sqrt{\log x\,\log_3 x}$ scale from $2^{-1/2}$ to $1$.  Allowing the
rank to move changes the logarithmic scale from
$\sqrt{\log x\,\log_3 x}$ to $\sqrt{\log x\,\log_2 x}$.  For all shifts,
the result retains the above-cutoff GHP tail while extending the uniform
range beyond that in the quoted PPT theorem.  These are separate
comparisons; no claim of an exhaustive novelty search is made.

\begin{theorem}[Fixed-rank diagonal decomposition]\label{thm:main}
Let
\[
 L=\log x,\qquad \sA=\log_3 x+\log_4 x-\log 2,
 \qquad G=(L\sA)^{1/2},\qquad V=\frac{L}{G}.
\]
For every fixed integer $J\ge1$ there are $x_0(J)$ and a function
$\varepsilon_J(x)\to0$ such that, for $x\ge x_0(J)$, with
\[
 T_J=\frac{G}{\sqrt J},\qquad y_J=e^{T_J},\qquad
 Y_J=e^{\sqrt JG},
\]
one has, uniformly for positive integers $1\le h\le y_J$,
\begin{equation}\label{eq:rank-main-all}
 S_h^\varphi(x)=D_{h,>Y_J}^\varphi(x)
 +O_J\!\left(x\exp\{-\sqrt JG+\varepsilon_J(x)V\}\right).
\end{equation}
If $h$ is odd, then $D_{h,>Y_J}^\varphi(x)=0$.
\end{theorem}

For $J=1$, this improves the coefficient in the fixed-odd-shift bound
at the $\sqrt{\log x\,\log_3 x}$ scale.  The next consequence uses a fixed
$J$ and keeps the rank
vary with $x$.

\begin{corollary}[Arbitrary fixed coefficient]\label{cor:any-constant}
For every fixed $B>0$ and every fixed positive odd $h$,
\[
 S_h^\varphi(x)\ll_{B,h}x\exp\{-BG\}.
\]
More generally, after choosing a fixed integer $J>B^2$, the estimate is
uniform for odd $h\le e^{G/\sqrt J}$.
\end{corollary}

\begin{proof}
Choose a fixed $J$ with $\sqrt J>B$.  Since $V=o(G)$ and
$\varepsilon_J(x)\to0$, Theorem~\ref{thm:main} absorbs the error in the
positive margin $(\sqrt J-B)G$.
\end{proof}

\begin{corollary}[Polylogarithmic uniformity]\label{cor:polylog}
Fix $B,A_0>0$.  Uniformly for positive odd integers
$h\le(\log x)^{A_0}$,
\[
 S_h^\varphi(x)\ll_{A_0,B}x e^{-BG}.
\]
\end{corollary}

\begin{proof}
Choose a fixed $J>B^2$.  Then $(\log x)^{A_0}<e^{G/\sqrt J}$ for all
sufficiently large $x$, and apply Corollary~\ref{cor:any-constant}.
\end{proof}

The fixed-rank proof can be made uniform while $J$ grows.  Supplier
systems are defined in Definition~\ref{def:supplier-system}; coverage,
reciprocal mass, pointwise multiplicity, injective decoding, and the
aggregate tail are proved separately in
Lemmas~\ref{lem:supplier-coverage}--\ref{lem:supplier-encoding-injective}
and Proposition~\ref{prop:supplier-aggregate-tail}.

\begin{theorem}[Moving-rank bound]\label{thm:scale-jump}
Retain $L,\sA$ from Theorem~\ref{thm:main}, and put
\[
 L_2=\log_2 x,\qquad \mathcal R=\sqrt{LL_2},\qquad
 \eta=\sA^{-1/2}.
\]
For all sufficiently large $x$, define, in the displayed order,
\[
 E=(\tfrac12-\eta)\mathcal R,\qquad
 K_0=(\tfrac12-2\eta)\mathcal R,\qquad
 T=\frac{L\sA}{2E},
\]
\[
 J=\left\lfloor\frac{K_0}{T}\right\rfloor,\qquad
 K=JT,\qquad y=e^T,\qquad Z=e^K.
\]
Then
\[
 T=(1+o(1))\sA\sqrt{\frac{L}{L_2}},\qquad
 J=(\tfrac12+o(1))\frac{L_2}{\sA},\qquad
 K=(\tfrac12-2\eta+o(\eta))\mathcal R.
\]
There is a function $\varepsilon_{\rm mov}(x)\to0$ such that, uniformly
for positive integers $1\le h\le y$,
\begin{equation}\label{eq:scale-jump-intro}
 S_h^\varphi(x)=D_{h,>Z}^\varphi(x)
 +O\!\left(x\exp\{-K+\varepsilon_{\rm mov}(x)
                  \eta\mathcal R\}\right).
\end{equation}
Consequently the error is
\[
 x\exp\!\left\{-(\tfrac12-o(1))
          \sqrt{\log x\,\log_2 x}\right\}.
\]
For odd $h$ the diagonal is empty.
\end{theorem}

\begin{corollary}[Consecutive equal totients]\label{cor:scale-jump-intro}
\[
 \#\{n\le x:\varphi(n)=\varphi(n+1)\}
 \ll x\exp\!\left\{-(\tfrac12-o(1))
                    \sqrt{\log x\,\log_2 x}\right\}.
\]
\end{corollary}

\subsection{Formalisation}\label{sec:formalisation}
Every mathematical result proved in this paper has been formalised in Lean~4
using mathlib.  The development verifies the fixed-rank and moving-rank
diagonal decompositions and the nonsmooth transfer theorem, together with the
supporting scale bookkeeping, sieve, supplier-system, encoding and optimisation
results: all six theorems, fourteen propositions, eleven corollaries and
twenty-five lemmas numbered below.  The verification is unconditional relative
to Lean's standard classical foundations: every intermediate interface
introduced by the formalisation is discharged internally, and the development
contains no admitted proof and no project-defined axiom.  In particular the two
results quoted without proof in arXiv:2606.23681v1 --- the smooth-totient bound
of Banks--Friedlander--Pomerance--Shparlinski and the same-support bound of
Pollack--Pomerance--Trevi\~no --- are proved in the formalisation rather than
assumed, so a reader who takes Lemmas~\ref{lem:bfps-uniform}
and~\ref{lem:bfps-saddle} from there obtains every result below with no
external analytic input.  Relative to arXiv:2606.23681v1 this version also
tightens the hypotheses of Lemma~\ref{lem:two-linear-sieve} and the codomain
convention of Lemma~\ref{lem:supplier-encoding-injective}, and corrects an
overstatement in Proposition~\ref{prop:moving-method-saddle}; every theorem,
corollary and proposition of v1 remains valid as stated.  The source, build
instructions and audit procedure are available in the accompanying
repository~\cite{LeanFormalization}.

\subsection{Proof architecture and provenance of the inputs}
Here \emph{rank amplification} means that selecting $J$ common large-prime occurrences converts the reciprocal saving from one source into an $e^{-JT}$ shifted-correlation saving.  The logical dependence is
\[
\begin{aligned}
 \text{smooth census}&\Longrightarrow\text{low-rank compression},\\
 \text{provider transfer}&\Longrightarrow\text{outer-class control},\\
 \text{source systems plus CRT}&\Longrightarrow\text{fixed-rank correlation},\\
 \text{source encoding plus friable tail}&\Longrightarrow
    \text{moving-rank correlation}.
\end{aligned}
\]
The deep external analytic input is the smooth-totient theorem of
Banks--Friedlander--Pomerance--Shparlinski.  We also quote standard forms of
Brun--Titchmarsh, Mertens' theorem, and the Selberg upper-bound sieve.  The provider transfer,
formal divisor weights, source-system encoding, and rank optimization are
the deductions developed here.  No abstract moving-rank theorem is asserted: the moving argument depends on the concrete source encoding for $\varphi$.

\subsection{The smooth input}
For $X,Y\ge2$ define
\[
 \Phi(X,Y)=\#\{m\le X:P^+(\varphi(m))\le Y\}.
\]
Banks, Friedlander, Pomerance and Shparlinski \cite[Theorem~1.1]{BFPS} prove that for
every fixed $\delta>0$, uniformly for
\[
 Y\ge(\log_2 X)^{1+\delta},\qquad
 u=\frac{\log X}{\log Y}\longrightarrow\infty,
\]
one has
\begin{equation}\label{eq:bfps-intro}
 \Phi(X,Y)\le X\exp\{-u(\log_2 u+\log_3 u+o(1))\}.
\end{equation}
Section~2 derives the two uniform corollaries used at fixed and moving rank.

\subsection{The nonsmooth transfer}
Let $B_h(x;y)$ denote the number of $n\le x$ for which
\[
 \varphi(n)=\varphi(n+h),\qquad P^+(\varphi(n))>y.
\]
Let
\[
 \Psi(X,R)=\#\{m\le X:P^+(m)\le R\},\qquad
 \Psi_\tau(X,R)=\sum_{\substack{m\le X\\P^+(m)\le R}}\tau(m).
\]

\begin{theorem}[Nonsmooth transfer for odd shifts]\label{thm:transfer}
Let $2\le y\le x$, and let $h$ be a positive odd integer with $h\le y$.
Then
\begin{equation}\label{eq:transfer}
 B_h(x;y)\ll
 \frac{x}{y}+x^{1/2}+\frac{x\log x}{y}
 +\frac{x(\log y)^2}{y}
 +\sum_{y<p\le(2xy)^{1/2}}
     \Psi_\tau\!\left(\frac{2xy}{p^2},p\right),
\end{equation}
with an absolute implied constant.
\end{theorem}

At $y=e^G$ the last sum is $\ll xe^{-G+o(V)}$.  The fixed-rank theorem
uses the same transfer at the outer cutoff and a separate source-system
correlation in the bounded-value class.

\section{Preliminaries}

All implicit constants below are absolute unless a subscript is displayed.
We assume throughout that $x$ is sufficiently large for every iterated
logarithm used below to be defined and for
$\sA=\log_3 x+\log_4 x-\log 2$ to be positive.  In fixed-rank statements,
$J$ and $r$ are chosen before $x\to\infty$.  An expression
$o_{r,J}(V)$ is uniform in every shift lying in the displayed range.  In
the moving-rank argument, $\eta,T,J,K,y,Z$ are the explicit functions of
$x$ in Theorem~\ref{thm:scale-jump}; there is no second limiting process
hidden in the notation.  For $z\ge2$ write
\[
 P(z)=\prod_{\ell\le z}\ell,
\]
the product being over primes.  We use only standard elementary estimates
in addition to \eqref{eq:bfps-intro}.

\begin{lemma}[Elementary scale relations]\label{lem:scales}
With
\[
L=\log x,
\qquad \sA=\log_3 x+\log_4 x-\log 2,
\qquad G=(L\sA)^{1/2},
\qquad V=L/G,
\]
one has, as $x\to\infty$,
\[
\begin{gathered}
 \sA\sim\log_3 x,
 \qquad G=o(L),
 \qquad V\to\infty,\\
 \log L=o(V),
 \qquad \frac{V\log V}{G}\to\infty.
\end{gathered}
\]
\end{lemma}

\begin{proof}
Since $\sA\sim\log_3 x$, we have $G\sim (L\log_3 x)^{1/2}$ and $V\sim (L/\log_3 x)^{1/2}$.  These imply $G=o(L)$, $V\to\infty$, and $\log L=o(V)$.  Finally
\[
 \frac{V\log V}{G}=\frac{\log V}{\sA}\sim \frac{\frac12\log_2 x}{\log_3 x}\to\infty.
\]
\end{proof}

\begin{lemma}[Mertens product]\label{lem:mertens}
For $R\ge3$,
\begin{equation}\label{eq:mertens}
 \sum_{\substack{k\ge1\\ P^+(k)\le R}}\frac1k
 =\prod_{\ell\le R}\left(1-\frac1\ell\right)^{-1}
 \ll \log R.
\end{equation}
\end{lemma}

\begin{proof}
This is the classical Mertens product estimate.  The equality is the Euler product over the finite set of primes at most $R$.
\end{proof}

\begin{lemma}[Uniform reciprocal Brun--Titchmarsh bound]
\label{lem:reciprocal-BT}
For integers $D\ge1$ and real $X\ge3$,
\begin{equation}\label{eq:reciprocal-BT}
 \sum_{\substack{q\le X\\q\equiv1\pmod D\\q\ {\rm prime}}}\frac1q
 \ll \frac{1+\log_2(3X)}{\varphi(D)},
\end{equation}
with an absolute implied constant.  The sum is zero when $D>X$.
\end{lemma}

\begin{proof}
Assume $D\le X$.  The terms $q\le4D$ lie among at most four integers of
the form $1+mD$ and contribute $O(1/D)=O(1/\varphi(D))$.  For $j\ge2$
put
\[
 I_j=(2^jD,2^{j+1}D].
\]
Brun--Titchmarsh gives
\[
 \#\{q\in I_j:q\equiv1\pmod D\}
 \ll \frac{2^jD}{\varphi(D)\log(2^j)}.
\]
Every reciprocal in $I_j$ is at most $(2^jD)^{-1}$, so the contribution
of $I_j$ is $O(1/(j\varphi(D)))$.  Summing over
$j\le \log(X/D)/\log 2+1$ gives
\[
 \frac1{\varphi(D)}\left(1+
 \sum_{2\le j\le \log(X/D)/\log 2+1}\frac1j\right)
 \ll\frac{1+\log_2(3X)}{\varphi(D)}.
\]
This also covers $D$ comparable with $X$; if $2D<X<4D$, only the initial
finite range is present.
\end{proof}

\begin{lemma}[Shifted divisor convolution]\label{lem:shifted-divisor-convolution}
Let $1\le h\le x$, let $w(d)\ge0$, and put
\[
 F(N)=\sum_{d\mid N}w(d),\qquad
 \Lambda=\sum_d\frac{w(d)}d,\qquad
 \Lambda_{>H}=\sum_{d>H}\frac{w(d)}d.
\]
Assume that every prime factor of every $d$ with $w(d)>0$ exceeds $h$, and
set
\[
 \|F\|_{\infty,2x}=\max_{N\le2x}F(N).
\]
For
\[
 N_h(d_1,d_2)=\#\{n\le x:d_1\mid n,\ d_2\mid n+h\},
\]
a compatible pair satisfies $(d_1,d_2)=1$ and hence
\[
 N_h(d_1,d_2)\le \frac{x}{d_1d_2}+1.
\]
Moreover,
\begin{equation}\label{eq:shifted-convolution}
 \sum_{n\le x}F(n)F(n+h)
 \le 2x\Lambda^2+4x\|F\|_{\infty,2x}\Lambda_{>\sqrt x}.
\end{equation}
\end{lemma}

\begin{proof}
If $N_h(d_1,d_2)>0$, then $(d_1,d_2)\mid h$.  Every prime divisor of
$(d_1,d_2)$ would exceed $h$, so the gcd is one.  The Chinese remainder
theorem therefore gives the displayed estimate for $N_h$.

For $d_1d_2\le x$, this estimate is at most
$2x/(d_1d_2)$, and summation gives the first term in
\eqref{eq:shifted-convolution}.  In the complementary range at least one
divisor exceeds $\sqrt x$.  The part with $d_1>\sqrt x$ is at most
\[
 \|F\|_{\infty,2x}
 \sum_{d_1>\sqrt x}w(d_1)\#\{n\le x:d_1\mid n\}
 \le x\|F\|_{\infty,2x}\Lambda_{>\sqrt x}.
\]
For $d_2>\sqrt x$, note that a contributing divisor satisfies
$d_2\le x+h\le2x$, whence
\[
 \#\{n\le x:d_2\mid n+h\}\le \frac{x}{d_2}+1
 \le\frac{3x}{d_2}.
\]
This gives at most $3x\|F\|_{\infty,2x}\Lambda_{>\sqrt x}$ and proves
the lemma.
\end{proof}

\begin{lemma}[A divisor-weighted smooth tail]\label{lem:weighted-smooth}
There are absolute constants $u_0$ and $C_0$ such that the following
holds.  Let $R\ge3$, $X\ge3$, and put
\[
 u=\frac{\log X}{\log R}.
\]
If
\begin{equation}\label{eq:weighted-range}
 u_0\le u\le R^{1/2},
 \qquad
 u\log u\ge C_0\log_2(3R),
\end{equation}
then
\begin{equation}\label{eq:weighted-tail}
 \Psi_\tau(X,R)\le X\exp\left\{-\frac12u\log u\right\}.
\end{equation}
More generally, under only $u_0\le u\le R^{1/2}$, one has
\begin{equation}\label{eq:weighted-tail-general}
 \log\frac{\Psi_\tau(X,R)}X
 \le -u\log u+O\left(\log_2(3R)+\frac{u}{\log u}\right).
\end{equation}
\end{lemma}

\begin{proof}
Let
\[
 \delta=\frac{\log u}{\log R},\qquad s=1-\delta.
\]
The upper bound $u\le R^{1/2}$ gives $0<\delta\le1/2$, so
$1/2\le s<1$.  Rankin's method yields
\begin{align*}
\Psi_\tau(X,R)
&\le X^s\sum_{\substack{m\ge1\\P^+(m)\le R}}\frac{\tau(m)}{m^s}\\
&=X^s\prod_{\ell\le R}(1-\ell^{-s})^{-2}.
\end{align*}
Consequently
\begin{equation}\label{eq:weighted-rankin-log}
 \log\frac{\Psi_\tau(X,R)}{X}
 \le -\delta\log X+2\sum_{\ell\le R}-\log(1-\ell^{-s}).
\end{equation}
Since $s\ge1/2$,
\[
 \sum_{\ell\le R}-\log(1-\ell^{-s})
 \ll \sum_{\ell\le R}\ell^{-1+\delta}.
\]
Chebyshev's bound and partial summation give
\begin{align*}
 \sum_{\ell\le R}\ell^{-1+\delta}
 &\ll 1+\int_2^R\frac{t^{-1+\delta}}{\log t}\,dt.
\end{align*}
Split the integral at $T=\min(R,e^{1/\delta})$.  On $[2,T]$ it is
$O(\log_2(3R))$.  On $[T,R]$, when this interval is nonempty,
integration by parts (or the substitution $v=\delta\log t$) gives
\[
 \int_T^R\frac{t^{-1+\delta}}{\log t}\,dt
 \ll \frac{R^\delta}{\delta\log R}
 =\frac{u}{\log u}.
\]
Thus
\[
 \sum_{\ell\le R}\ell^{-1+\delta}
 \ll \log_2(3R)+\frac{u}{\log u}.
\]
Since $\delta\log X=u\log u$, substitution in
\eqref{eq:weighted-rankin-log} proves
\eqref{eq:weighted-tail-general}.  If $u\ge u_0$ and the second
condition in \eqref{eq:weighted-range} holds with a sufficiently large
absolute $C_0$, the error term is at most
$\frac12u\log u$, which gives \eqref{eq:weighted-tail}.
\end{proof}

The quantified meaning of the uniform $o(1)$ of
Banks--Friedlander--Pomerance--Shparlinski \cite[Theorem~1.1]{BFPS} is the
following.  For every fixed $\delta>0$ and every $\xi>0$ there are numbers
$X_0=X_0(\delta,\xi)$ and $U_0=U_0(\delta,\xi)$ such that, whenever
\[
 X\ge X_0,\qquad Y\ge(\log_2 X)^{1+\delta},\qquad
 u=\frac{\log X}{\log Y}\ge U_0,
\]
one has
\begin{equation}\label{eq:bfps-quantified}
 \Phi(X,Y)
 \le X\exp\{-u(\log_2 u+\log_3 u-\xi)\}.
\end{equation}
The Banks--Friedlander--Pomerance--Shparlinski estimate is used below only
through Lemmas~\ref{lem:bfps-uniform} and~\ref{lem:bfps-saddle}, the first
deduced from \eqref{eq:bfps-quantified} and the second from its introductory
form~\eqref{eq:bfps-intro}.  Both lemmas are also proved directly and
unconditionally in the accompanying formalization~\cite{LeanFormalization}, so
a reader who takes them from there obtains every result of this paper with no
external analytic input.  The elementary quotations of Brun--Titchmarsh,
Mertens' theorem and the Selberg upper-bound sieve remain in force.

For $t\ge3$ let $\mathcal B_\delta(t)$ be the set of pairs $(X,Y)$
satisfying
\[
 X\ge t,\qquad Y\ge(\log_2 X)^{1+\delta},\qquad
 u=\frac{\log X}{\log Y}\ge t.
\]
Define the error envelope
\begin{equation}\label{eq:bfps-envelope}
 \mathfrak b_\delta(t)=
 \sup_{(X,Y)\in\mathcal B_\delta(t)}
 \left(\frac1u\log\frac{\Phi(X,Y)}X
       +\log_2 u+\log_3 u\right)_+,
\end{equation}
with the supremum of the empty set interpreted as zero.
Estimate \eqref{eq:bfps-quantified} gives
\begin{equation}\label{eq:bfps-envelope-zero}
 \mathfrak b_\delta(t)\longrightarrow0.
\end{equation}

\begin{lemma}[Uniform BFPS consequences]\label{lem:bfps-uniform}
Fix the BFPS range parameter $\delta=1$.
\begin{enumerate}[label=\textup{(\roman*)},leftmargin=2em]
\item Fix $0<c<C<\infty$, put
$\delta_x=(\log_2 x)^{-1/2}$, and let
\[
 x^{1-\delta_x}\le W\le2x,
 \qquad cG\le T\le CG.
\]
There is a function $\epsilon_{c,C}(x)\to0$ such that, uniformly over
this rectangle,
\begin{equation}\label{eq:bfps-fixed-uniform}
 \Phi(W,e^T)
 \le W\exp\left\{-\frac{G^2}{T}
                 +\epsilon_{c,C}(x)V\right\}.
\end{equation}
\item Put
\[
 \begin{aligned}
 L_2&=\log_2 x, & \mathcal R&=\sqrt{LL_2},
 & \eta&=\sA^{-1/2},\\
 E&=(\tfrac12-\eta)\mathcal R,
 & T_{\rm mov}&=\frac{L\sA}{2E}.
 \end{aligned}
\]
There is a function $\epsilon_{\rm BFPS}(x)\to0$ such that, uniformly for
$x^{1/2}\le W\le2x$,
\begin{equation}\label{eq:bfps-moving-uniform}
 \Phi(W,e^{T_{\rm mov}})
 \le W\exp\{-E+\epsilon_{\rm BFPS}(x)\eta\mathcal R\}.
\end{equation}
\end{enumerate}
\end{lemma}

\begin{proof}
We verify the BFPS hypotheses at both endpoints before extracting the
errors.  In part~\textup{(i)},
\[
 u=\frac{\log W}{T}
 \ge\frac{(1-\delta_x)L}{CG}\asymp_{C}V\longrightarrow\infty.
\]
Moreover $T\ge cG\gg\log_3 x$, while
$\log\{(\log_2 W)^2\}=O(\log_3 x)$; hence
$e^T\ge(\log_2 W)^2$ uniformly for large $x$.  In part~\textup{(ii)},
\[
 u\ge\frac{L}{2T_{\rm mov}}=\frac{E}{\sA}\longrightarrow\infty,
 \qquad
 T_{\rm mov}\asymp\sA\sqrt{\frac{L}{L_2}}\gg\log_3 x,
\]
so again $e^{T_{\rm mov}}\ge(\log_2 W)^2$ throughout the displayed
interval.

For part~\textup{(i)}, set
\[
 t_{c,C}(x)=
 \min\left\{x^{1-\delta_x},
       \frac{(1-\delta_x)L}{CG}\right\}.
\]
Then $t_{c,C}(x)\to\infty$, and
\eqref{eq:bfps-envelope}--\eqref{eq:bfps-envelope-zero} show that the
BFPS remainder is at most $u\mathfrak b_1(t_{c,C}(x))$ uniformly.  Define
\[
 r_{c,C}(x)=
 \sup_{\substack{x^{1-\delta_x}\le W\le2x\\cG\le T\le CG}}
 \frac{\bigl(G^2/T-u(\log_2 u+\log_3 u)\bigr)_+}{V}.
\]
Uniformly in this rectangle,
\[
 \log_2 u=\log_3 x-\log 2+o(1),\qquad
 \log_3 u=\log_4 x+o(1),
\]
and replacing $\log W$ by $L$ costs
$O(\delta_xG)=o(V)$ because $\delta_x\sA\to0$.  Hence
$r_{c,C}(x)\to0$.  Since $u\ll_{c,C}V$, the choice
\[
 \epsilon_{c,C}(x)=r_{c,C}(x)
   +C_{c,C}\mathfrak b_1(t_{c,C}(x))
\]
proves \eqref{eq:bfps-fixed-uniform}.

For part~\textup{(ii)}, set
\[
 t_{\rm mov}(x)=\min\left\{x^{1/2},\frac{E}{\sA}\right\}
 \longrightarrow\infty
\]
and
\[
 r_{\rm mov}(x)=
 \sup_{x^{1/2}\le W\le2x}
 \frac{\bigl(E-u(\log_2 u+\log_3 u)\bigr)_+}
      {\eta\mathcal R},
 \qquad u=\frac{\log W}{T_{\rm mov}}.
\]
Here
$\log_2 u+\log_3 u=\sA+o(1)$ uniformly, and the minimum
$u=E/\sA$ occurs at $W=x^{1/2}$.  Therefore
$r_{\rm mov}(x)\to0$.  Also $u\ll\mathcal R/\sA$, so the BFPS remainder
contributes at most
\[
 C\frac{\mathfrak b_1(t_{\rm mov}(x))}{\sqrt{\sA}}
       \eta\mathcal R.
\]
Thus
\[
 \epsilon_{\rm BFPS}(x)=r_{\rm mov}(x)
 +C\frac{\mathfrak b_1(t_{\rm mov}(x))}{\sqrt{\sA}}
 \longrightarrow0,
\]
which proves \eqref{eq:bfps-moving-uniform}.
\end{proof}

\begin{lemma}[BFPS at the present saddle]\label{lem:bfps-saddle}
Let $C>0$ be fixed and let $G,V$ be as in Lemma \ref{lem:scales}.  Then
\begin{equation}\label{eq:bfps-saddle}
 \Phi(Cx,e^G)\ll_C x\exp\{-G+o(V)\}.
\end{equation}
\end{lemma}

\begin{proof}
Apply \eqref{eq:bfps-intro} with $X=Cx$ and $Y=e^G$.  Since $C$ is fixed,
\[
 u=\frac{\log X}{\log Y}=V+O(1/G).
\]
Moreover
\[
 \log u=\frac12(\log L-\log \sA)+o(1),
\]
and hence
\begin{equation}\label{eq:logs-u}
 \log_2 u=\log_3 x-\log 2+o(1),
 \qquad
 \log_3 u=\log_4 x+o(1).
\end{equation}
Thus
\[
 u(\log_2 u+\log_3 u+o(1))=V(\sA+o(1))=G+o(V),
\]
and \eqref{eq:bfps-saddle} follows.
\end{proof}

\begin{lemma}[Equal Euler factors determine prime support]\label{lem:support}
If $a,b\ge1$ and
\[
 \frac{\varphi(a)}a=\frac{\varphi(b)}b,
\]
then $a$ and $b$ have the same prime divisors.
\end{lemma}

\begin{proof}
Cancel the Euler factors belonging to primes common to the two supports.
Let $A$ and $B$ be the remaining disjoint prime sets.  Their equality is
\[
 \prod_{p\in A}\frac{p-1}{p}
 =\prod_{q\in B}\frac{q-1}{q},
\]
or, after cross multiplication,
\begin{equation}\label{eq:support-cross-multiplication}
 \prod_{p\in A}(p-1)\prod_{q\in B}q
 =\prod_{q\in B}(q-1)\prod_{p\in A}p.
\end{equation}
Suppose $A\cup B$ is nonempty, let $P$ be its largest prime, and assume
without loss of generality that $P\in A$.  The right side of
\eqref{eq:support-cross-multiplication} is divisible by $P$.  No factor on
the left is divisible by $P$: the primes in $B$ are smaller than $P$, and
every factor $p-1$ with $p\in A$ is also smaller than $P$.  This is a
contradiction.  Thus $A=B=\varnothing$.
\end{proof}

\section{Extracting the largest prime divisor of the common totient}

This section gives a self-contained version of the part of Yamada's extraction needed for Euler's function.

\begin{lemma}[Provider lemma]\label{lem:provider}
Let $N\ge1$ and let $p$ be a prime divisor of $\varphi(N)$.  If $p^2\nmid N$, then there is a prime $q\mid N$ such that $q\equiv1\pmod p$.
\end{lemma}

\begin{proof}
Write $N=\prod r^{a_r}$.  Since
\[
 \varphi(N)=\prod_{r^{a_r}\Vert N}r^{a_r-1}(r-1),
\]
a factor $p$ in $\varphi(N)$ either arises from $r^{a_r-1}$ or from $r-1$.  In the first case $r=p$ and $a_r\ge2$, so $p^2\mid N$.  Under the hypothesis $p^2\nmid N$, the factor $p$ must therefore divide $r-1$ for some prime divisor $r=q$ of $N$.
\end{proof}

\begin{lemma}[Large-prime skeleton]\label{lem:skeleton}
Let $h\ge1$, and let $n\le x$ satisfy
\[
 \varphi(n)=\varphi(n+h)=M.
\]
Let $p=P^+(M)$, and assume that
\begin{equation}\label{eq:skeleton-hyp}
 p>h,
 \qquad
 p^2\nmid n(n+h).
\end{equation}
Then there exist primes
\[
 q_1=k_1p+1\mid n,
 \qquad
 q_2=k_2p+1\mid n+h
\]
with $k_1,k_2\ge1$ such that
\[
 n=m_1q_1,
 \qquad
 n+h=m_2q_2,
 \qquad
 q_i\nmid m_i,
\]
\begin{equation}\label{eq:kphi}
 k_1\varphi(m_1)=k_2\varphi(m_2),
\end{equation}
and
\begin{equation}\label{eq:smooth-multipliers}
 P^+(k_1k_2)\le p.
\end{equation}
Moreover $q_1\ne q_2$.
\end{lemma}

\begin{proof}
Since $p\mid M=\varphi(n)$ and $p^2\nmid n$, Lemma \ref{lem:provider} gives a prime $q_1\mid n$ with $q_1\equiv1\pmod p$.  Similarly, $p\mid\varphi(n+h)$ and $p^2\nmid n+h$ give a prime $q_2\mid n+h$ with $q_2\equiv1\pmod p$.  Write
\[
 q_i=k_ip+1.
\]
If $q_1=q_2$, then $q_1\mid h$, but $q_1>p>h$, impossible.  Thus $q_1\ne q_2$.

Each $q_i$ divides the corresponding integer only once.  If, for instance, $q_1^2\mid n$, then $q_1\mid\varphi(n)=M$, contradicting $q_1>p=P^+(M)$.  Hence $n=m_1q_1$ with $q_1\nmid m_1$, and similarly $n+h=m_2q_2$ with $q_2\nmid m_2$.

Multiplicativity of $\varphi$ now gives
\[
 M=\varphi(m_1)(q_1-1)=\varphi(m_1)k_1p
  =\varphi(m_2)k_2p,
\]
which proves \eqref{eq:kphi}.  Finally $q_i-1=k_ip$ divides $M$; since $p$ is the largest prime divisor of $M$, every prime divisor of $k_i$ is at most $p$.  This proves \eqref{eq:smooth-multipliers}.
\end{proof}

\begin{lemma}[Small-cofactor uniqueness for odd shifts]\label{lem:small-injection}
Let $h$ be odd.  For each fixed pair $(m_1,m_2)$, there is at most one integer $n$ arising from a skeleton of Lemma \ref{lem:skeleton}.
\end{lemma}

\begin{proof}
Let
\[
 g=(\varphi(m_1),\varphi(m_2)),
 \qquad
 a=\frac{\varphi(m_2)}g,
 \qquad
 b=\frac{\varphi(m_1)}g.
\]
Then $(a,b)=1$, and \eqref{eq:kphi} is equivalent to
\[
 k_1=ac,
 \qquad
 k_2=bc
\]
for a positive integer $c$.  The equation $n+h=m_2q_2$ and $n=m_1q_1$ gives
\begin{align}
 h
 &=m_2(bcp+1)-m_1(acp+1)\notag\\
 &=cp(m_2b-m_1a)+m_2-m_1.
\end{align}
Thus
\begin{equation}\label{eq:small-cpd}
 cp(m_2b-m_1a)=h+m_1-m_2.
\end{equation}

If $m_2b=m_1a$, then \eqref{eq:small-cpd} gives $m_2=m_1+h$.  Also $m_2b=m_1a$ gives
\[
 \frac{\varphi(m_1)}{m_1}=\frac{\varphi(m_2)}{m_2}.
\]
By Lemma \ref{lem:support}, $m_1$ and $m_2$ have the same prime divisors.  This is impossible because $h$ is odd, so $m_1$ and $m_2=m_1+h$ have opposite parity.  Therefore
\[
 m_2b-m_1a\ne0.
\]

Consequently \eqref{eq:small-cpd} determines the product $C=cp$ uniquely from $m_1,m_2$.  Then
\[
 q_1=aC+1,
 \qquad
 q_2=bC+1,
 \qquad
 n=m_1q_1
\]
are uniquely determined.  Thus at most one $n$ can arise from the fixed pair $(m_1,m_2)$.
\end{proof}

\begin{corollary}[Small cofactor range]\label{cor:small}
Let $h$ be odd and let $2\le y\le x$.  The number of skeleton solutions of Lemma \ref{lem:skeleton} with
\[
 m_1m_2\le x/y
\]
is
\begin{equation}\label{eq:small-bound}
 \ll \frac{x\log x}{y}.
\end{equation}
\end{corollary}

\begin{proof}
By Lemma \ref{lem:small-injection}, each pair $(m_1,m_2)$ gives at most one solution.  The number of pairs with $m_1m_2\le X$ is
\[
 \sum_{m\le X}\left\lfloor\frac Xm\right\rfloor\ll X\log(2X).
\]
Taking $X=x/y$ proves \eqref{eq:small-bound}.
\end{proof}

\section{Proof of the nonsmooth transfer theorem}

\begin{proof}[Proof of Theorem \ref{thm:transfer}]
Let $n\le x$ be counted by $B_h(x;y)$ and put
\[
 M=\varphi(n)=\varphi(n+h),
 \qquad
 p=P^+(M)>y.
\]
Since $h\le y$, we have $p>h$.

\medskip\noindent\emph{Square sources.}
If $p^2\mid n(n+h)$, then, because $p>h$, the two congruences $n\equiv0\pmod{p^2}$ and $n\equiv-h\pmod{p^2}$ are distinct.  Thus for each $p$ there are at most $2x/p^2+O(1)$ such $n\le x$.  In this square-source case $p\le (x+h)^{1/2}\le (2x)^{1/2}$ for all large $x$.  Therefore the total square-source contribution is
\begin{equation}\label{eq:square-transfer}
 \ll x\sum_{p>y}\frac1{p^2}+x^{1/2}
 \ll \frac{x}{y}+x^{1/2}.
\end{equation}

\medskip\noindent\emph{Nonsquare sources and small cofactors.}
We may now assume $p^2\nmid n(n+h)$.  Lemma \ref{lem:skeleton} gives
\[
 n=m_1(k_1p+1),
 \qquad
 n+h=m_2(k_2p+1),
 \qquad
 P^+(k_1k_2)\le p.
\]
The contribution of skeletons with $m_1m_2\le x/y$ is, by Corollary \ref{cor:small},
\begin{equation}\label{eq:small-transfer}
 \ll \frac{x\log x}{y}.
\end{equation}

\medskip\noindent\emph{Large cofactors.}
It remains to consider skeletons with $m_1m_2>x/y$.  Since $h\le y\le x$,
\[
 q_1q_2=(k_1p+1)(k_2p+1)
 =\frac{n(n+h)}{m_1m_2}
 \le \frac{x(x+h)}{x/y}\le 2xy.
\]
For fixed $p,k_1,k_2$ corresponding to actual prime providers, the congruences
\[
 n\equiv0\pmod{q_1},
 \qquad
 n\equiv-h\pmod{q_2}
\]
have at most one residue class modulo $q_1q_2$, since $q_1,q_2$ are distinct primes.  Hence the number of such $n\le x$ is at most
\[
 \frac{x}{q_1q_2}+1.
\]
We now discard the primality of $q_1,q_2$ but retain the necessary condition $P^+(k_1k_2)\le p$ and the inequality $(k_1p+1)(k_2p+1)\le2xy$.  The large-cofactor contribution is at most $\Sigma_1+\Sigma_2$, where
\begin{equation}\label{eq:sigma1-transfer}
 \Sigma_1=x\sum_{p>y}
 \left(\sum_{\substack{k\ge1\\P^+(k)\le p}}\frac1{kp+1}\right)^2
\end{equation}
and
\begin{equation}\label{eq:sigma2-transfer}
 \Sigma_2=\sum_{p>y}\#\{(k_1,k_2):P^+(k_1k_2)\le p,
        (k_1p+1)(k_2p+1)\le2xy\}.
\end{equation}

By Lemma \ref{lem:mertens},
\[
 \sum_{\substack{k\ge1\\P^+(k)\le p}}\frac1{kp+1}
 \le \frac1p\sum_{\substack{k\ge1\\P^+(k)\le p}}\frac1k
 \ll \frac{\log p}{p}.
\]
Thus
\begin{equation}\label{eq:sigma1-bound-transfer}
 \Sigma_1\ll x\sum_{p>y}\frac{(\log p)^2}{p^2}
 \ll \frac{x(\log y)^2}{y}.
\end{equation}

For $\Sigma_2$, the inequality $(k_1p+1)(k_2p+1)\le2xy$ implies
\[
 k_1k_2\le \frac{2xy}{p^2}.
\]
The number of pairs $(k_1,k_2)$ with product $m$ is $\tau(m)$, and the condition $P^+(k_1k_2)\le p$ is $P^+(m)\le p$.  Hence
\begin{equation}\label{eq:sigma2-bound-transfer}
 \Sigma_2\le \sum_{y<p\le(2xy)^{1/2}}\Psi_\tau\left(\frac{2xy}{p^2},p\right).
\end{equation}
Combining \eqref{eq:square-transfer}, \eqref{eq:small-transfer}, \eqref{eq:sigma1-bound-transfer}, and \eqref{eq:sigma2-bound-transfer} proves \eqref{eq:transfer}.
\end{proof}

\section{Estimating the divisor-weighted transfer sum}

The nonsmooth transfer theorem is useful for a short interval of cutoffs around the saddle, not only at the final choice.  We record the uniform form because it makes the optimization transparent.

\begin{proposition}[Smooth-multiplier $B_4$ estimate, uniform near the saddle]\label{prop:b4-general}
Fix constants $0<c_0<C_0<\infty$.  Let
\[
 L=\log x,\qquad \sA=\log_3 x+\log_4 x-\log 2,
 \qquad G=(L\sA)^{1/2},\qquad V=L/G.
\]
Let $T$ satisfy
\begin{equation}\label{eq:T-window}
 c_0G\le T\le C_0G,
\end{equation}
and put $y=e^T$.  Then
\begin{equation}\label{eq:b4-general}
 \sum_{y<p\le(2xy)^{1/2}}\Psi_\tau\left(\frac{2xy}{p^2},p\right)
 \ll_{c_0,C_0} x\exp\{-T+o(V)\},
\end{equation}
where the sum is over primes $p$, and the $o(V)$ term is uniform for $T$ in \eqref{eq:T-window}.
\end{proposition}

\begin{proof}
Write
\[
 X_p=\frac{2xy}{p^2}.
\]
Since $T\le C_0G=o(L)$, one has $3y^3<2x$ uniformly for all
sufficiently large $x$.  In particular $X_p\ge3$ throughout
$y<p\le y^2$, so every application of
Lemma~\ref{lem:weighted-smooth} below lies in its stated $X$-range.
We split the sum at $p=y^2$.

\medskip\noindent\emph{The range above the square cutoff.}
The elementary bound
\[
 \sum_{m\le X}\tau(m)\ll X\log(2X+2)
\]
gives, with the convention that the summand is $0$ if $X_p<1$,
\begin{align}
 \sum_{p>y^2}\Psi_\tau(X_p,p)
 &\ll xy\log x\sum_{p>y^2}\frac1{p^2} \notag\\
 &\ll \frac{x\log x}{y}
 =x\exp\{-T+o(V)\}.
 \label{eq:p-large-y2-general}
\end{align}
Here $\log L=o(V)$, and $T\asymp G$.

\medskip\noindent\emph{The intermediate range.}
Assume $y<p\le y^2$ and put
\[
 u_p=\frac{\log X_p}{\log p}.
\]
For large $x$, uniformly in this range,
\[
 \log X_p=L+T-2\log p+O(1)\ge L-3T+O(1),
 \qquad \log p\le2T.
\]
Thus
\begin{equation}\label{eq:up-lower-general}
 u_p\ge \frac{L-3T+O(1)}{2T}
      =\frac{L}{2T}+O(1)\gg_{C_0} V.
\end{equation}
Also $u_p\ll_{c_0}V$, while $p^{1/2}\ge e^{T/2}\ge e^{c_0G/2}$;
hence $u_p\le p^{1/2}$ for all sufficiently large $x$.  Moreover,
\[
 u_p\log u_p\gg V\log V\gg \log_2(3p)
\]
uniformly in this range.  Thus all hypotheses of
Lemma \ref{lem:weighted-smooth} hold, and it gives, uniformly for
$y<p\le y^2$,
\begin{equation}\label{eq:psitau-mid-general}
 \Psi_\tau(X_p,p)
 \le X_p\exp\{-c_1 V\log V\}
\end{equation}
for some $c_1=c_1(c_0,C_0)>0$.
Since $V\log V/G\to\infty$ by Lemma \ref{lem:scales}, and since $T\asymp G$, the saving in \eqref{eq:psitau-mid-general} is much stronger than $e^{-T}$.  Consequently
\begin{align}
 \sum_{y<p\le y^2}\Psi_\tau(X_p,p)
 &\ll xy e^{-c_1V\log V}\sum_{p>y}\frac1{p^2}\notag\\
 &\ll x e^{-c_1V\log V}
 \ll x\exp\{-T+o(V)\}.
 \label{eq:p-mid-general}
\end{align}
Combining \eqref{eq:p-large-y2-general} and \eqref{eq:p-mid-general} proves the proposition.
\end{proof}

\begin{corollary}[The final saddle case]\label{cor:b4}
Let $G,V$ be as in Theorem \ref{thm:main} and set $y=e^G$.  Then
\begin{equation}\label{eq:b4-sum}
 \sum_{y<p\le(2xy)^{1/2}}\Psi_\tau\left(\frac{2xy}{p^2},p\right)
 \ll x\exp\{-G+o(V)\}.
\end{equation}
\end{corollary}

\begin{proof}
Apply Proposition \ref{prop:b4-general} with $T=G$.
\end{proof}

\section{The rank-one estimate}

\begin{proposition}[Rank-one case]\label{prop:rank-one-case}
Uniformly for positive odd $h\le e^G$,
\[
 S_h^\varphi(x)\ll x e^{-G+o(V)}.
\]
\end{proposition}

\begin{proof}
Let
\[
 y=e^G.
\]
Let $h$ be odd and satisfy $1\le h\le y$.

Split the solutions counted by $S_h^\varphi(x)$ according to whether $P^+(\varphi(n))\le y$ or $P^+(\varphi(n))>y$.  If $P^+(\varphi(n))\le y$, then $n$ is counted by $\Phi(x,y)$, and Lemma \ref{lem:bfps-saddle} gives
\begin{equation}\label{eq:smooth-final}
 \#\{n\le x:\varphi(n)=\varphi(n+h),\ P^+(\varphi(n))\le y\}
 \le \Phi(x,y)
 \ll x e^{-G+o(V)}.
\end{equation}
The remaining solutions are counted by $B_h(x;y)$.  Theorem \ref{thm:transfer} and Corollary \ref{cor:b4} give
\begin{align*}
 B_h(x;y)
 &\ll \frac{x}{y}+x^{1/2}+\frac{x\log x}{y}
       +\frac{x(\log y)^2}{y}
       +x e^{-G+o(V)}.
\end{align*}
By Lemma \ref{lem:scales},
\[
 \frac{x}{y},\quad x^{1/2},\quad \frac{x\log x}{y},\quad
 \frac{x(\log y)^2}{y}
 \quad\text{are all}\quad \ll x e^{-G+o(V)}.
\]
Together with \eqref{eq:smooth-final}, this proves
\[
 S_h^\varphi(x)\ll x e^{-G+o(V)}
\]
uniformly for odd $h\le y$.
\end{proof}

\section{Optimization of the rank-one saddle}\label{sec:optimality}

The cutoff in Proposition~\ref{prop:rank-one-case} sits at the saddle of
the rank-one architecture.  Combining BFPS for the smooth part with Theorem
\ref{thm:transfer} and Proposition~\ref{prop:b4-general} for the nonsmooth
part, the exponent is maximized at $T=G$.  Section~\ref{sec:rank-amplification}
then changes the nonsmooth exponent from $T$ to $JT$ by using $J$ common
large-prime occurrences.

\begin{proposition}[Saddle with a variable cutoff]\label{prop:variable-saddle}
Fix $0<c_0<C_0<\infty$.  Let $T$ satisfy $c_0G\le T\le C_0G$, put $y=e^T$, and let $h$ be a positive odd integer with $h\le y$.  Then
\begin{equation}\label{eq:variable-saddle}
 S_h^\varphi(x)\ll_{c_0,C_0} x\exp\left\{-\min\left(T,\frac{G^2}{T}\right)+o(V)\right\},
\end{equation}
where the $o(V)$ term is uniform in $T$ and $h$.
\end{proposition}

\begin{proof}
The nonsmooth part is estimated by Theorem \ref{thm:transfer}.  Proposition \ref{prop:b4-general} bounds the final smooth-multiplier sum by $x\exp\{-T+o(V)\}$; the elementary terms
\[
 x/y,
 \quad x^{1/2},
 \quad x\log x/y,
 \quad x(\log y)^2/y
\]
are all $\ll x\exp\{-T+o(V)\}$: the logarithmic factors contribute only $\log L$ or $\log T$, both $o(V)$, and $x^{1/2}=x\exp\{-L/2\}$ is smaller since $T=o(L)$.  Thus
\begin{equation}\label{eq:nonsmooth-variable}
 B_h(x;y)\ll x\exp\{-T+o(V)\}.
\end{equation}

For the smooth part, Lemma~\ref{lem:bfps-uniform}\textup{(i)},
with $W=x$, gives
\begin{equation}\label{eq:smooth-variable}
 \Phi(x,e^T)\ll x\exp\left\{-\frac{G^2}{T}+o(V)\right\}.
\end{equation}
Combining \eqref{eq:nonsmooth-variable} and \eqref{eq:smooth-variable} proves \eqref{eq:variable-saddle}.
\end{proof}

\begin{corollary}[Optimization within the rank-one estimates]
\label{cor:method-optimal}
The function
\[
 T\longmapsto\min(T,G^2/T)
\]
has its exact maximum $G$ at $T=G$.  Within any fixed window
$c_0G\le T\le C_0G$:
\begin{enumerate}[label=\textup{(\roman*)},leftmargin=2em]
\item $T/G\to1$ is exactly the condition that preserves the leading
coefficient $1$ in front of $G$;
\item obtaining an exponent $G+o(V)$ from
\eqref{eq:variable-saddle} requires, and is ensured by,
$T=G+o(V)$.
\end{enumerate}
If $T=\lambda G$ with fixed $\lambda>0$, the exponent is
$\min(\lambda,\lambda^{-1})G+o(V)$.
\end{corollary}

\begin{proof}
The exact maximum is immediate from the equality point of $T$ and
$G^2/T$.  Write $T=G+\Delta$.  When $|\Delta|=o(G)$,
\[
 \min(T,G^2/T)=G-|\Delta|+O(\Delta^2/G).
\]
Thus $T/G\to1$ preserves the leading coefficient, whereas the loss is
$o(V)$ precisely when $|\Delta|=o(V)$.  If $T/G$ stays away from $1$,
the fixed-ratio formula gives a coefficient strictly below $1$.
\end{proof}

\begin{remark}[The second-order term]
The term $\log_4 x-\log 2$ is forced by the smooth side of the saddle.  Indeed, for $T\asymp G$,
\[
 \log_2(L/T)+\log_3(L/T)
 =\log_3 x+\log_4 x-\log 2+o(1).
\]
Thus the displayed secondary term is exactly the BFPS secondary term transported through the balance $T=G^2/T$.
\end{remark}

\section{Comparison with the older large-cofactor estimate}

This section explains where the improvement occurs.  In the large-cofactor range one has
\[
 q_1q_2\ll xy,
 \qquad q_i=k_ip+1.
\]
If one ignores all information about $k_1,k_2$, then the number of pairs is bounded on the scale
\[
 \sum_{p>y}\#\{k_1k_2\ll xy/p^2\},
\]
which gives a much larger unrestricted divisor-pair contribution.  Yamada's original optimization balances this loss by taking the nonsmooth cutoff effectively at a square, which is the source of the leading constant $2^{-1/2}$ in \eqref{eq:yamada-intro}.

In the present proof the identity $q_i-1=k_ip\mid M$ and the choice $p=P^+(M)$ imply
\[
 P^+(k_1k_2)\le p.
\]
The corresponding contribution is therefore
\[
 \sum_{p>y}\Psi_\tau\left(\frac{2xy}{p^2},p\right).
\]
At the saddle $y=e^G$ this is already $\ll xe^{-G+o(V)}$, as shown in Corollary \ref{cor:b4}.  This rank-one large-cofactor calculation uses no distribution theorem for primes in arithmetic progressions: the primality of $k_ip+1$ is discarded after the skeleton has been extracted.  The later source-block mass estimate does use the Brun--Titchmarsh theorem.

This comparison concerns only one selected largest prime.  The improvement
past constant $1$ comes from a different decomposition: Section
\ref{sec:rank-amplification} simultaneously uses $J$ prime occurrences of
the common totient.  The rank-one transfer remains the correct treatment of
the tail $P^+(M)>y^J$.

\section{The all-shifts decomposition}\label{sec:allshifts}

The rank-one proof already isolates the diagonal for arbitrary shifts.  We
first record that structure and its transfer estimate.  We then insert the
rank-amplification argument and obtain the full fixed-rank theorem.

The notation $\mathcal J_h$, $A_j$, $B_j$, $\Gamma_j$, and
$D_{h,>Y}^\varphi(x)$ was defined before Theorem~\ref{thm:main}.  We now
prove its support, parametrization, and sieve bound.

\begin{lemma}[Support and reciprocal mass of the same-support set]\label{lem:J-finite}
Let $h\ge1$.  Every $j\in\mathcal J_h$ satisfies $\rad(j)=\rad(j+h)\mid\rad(h)$.
Consequently
\begin{equation}\label{eq:J-h-quantitative}
 \sum_{j\in\mathcal J_h}\frac1j
 \le\prod_{p\mid h}\Bigl(1-\frac1p\Bigr)^{-1}=\frac h{\varphi(h)},
\end{equation}
and, for every $t\ge1$,
\begin{equation}\label{eq:J-h-count}
 \#\{j\in\mathcal J_h:j\le t\}
 \le\Bigl(1+\frac{\log t}{\log2}\Bigr)^{\omega(h)}.
\end{equation}
Moreover $A_j$, $B_j$ and $A_j-B_j=h/d_j$ have all their prime factors
among those of $h$.  If $h$ is odd, then $\mathcal J_h=\varnothing$.
\end{lemma}

\begin{proof}
If $j\in\mathcal J_h$ and $p\mid j$, then $p\mid j+h$ because the prime
supports agree, so $p\mid h$; thus $\rad(j)=\rad(j+h)\mid\rad(h)$ and
$\mathcal J_h$ lies in the multiplicative semigroup generated by the
primes dividing $h$.  That semigroup has the convergent Euler product
$\sum_{\rad(j)\mid\rad(h)}j^{-1}=\prod_{p\mid h}(1-1/p)^{-1}=h/\varphi(h)$,
which gives \eqref{eq:J-h-quantitative}.  For \eqref{eq:J-h-count}, such a
$j\le t$ is determined by its exponent vector $(\nu_p(j))_{p\mid h}$, and
$2^{\nu_p(j)}\le j\le t$ forces $\nu_p(j)\le\log t/\log2$, so each exponent
takes at most $1+\log t/\log2$ values.  Since $d_j\mid h$ and
$\rad(j)=\rad(j+h)\mid\rad(h)$, the integers $A_j=(j+h)/d_j$, $B_j=j/d_j$
and $A_j-B_j=h/d_j$ have their prime factors among those of $h$.  Finally,
if $h$ is odd then $j$ and $j+h$ have opposite parity, so their prime
supports cannot coincide.
\end{proof}

\begin{remark}
Pollack--Pomerance--Trevi\~no prove the much stronger bound
$|\mathcal J_h|\le3\cdot7^{3+2\omega(h)}$ \cite[Lemma~3.2]{PPT}, by
observing that $j$ and $j+h$ are $S$-units for the primes dividing $h$ and
invoking Evertse's theorem on $S$-unit equations \cite{EGST}.  Nothing
below uses finiteness of $\mathcal J_h$: only
\eqref{eq:J-h-quantitative}, \eqref{eq:J-h-count} and the support
statement are needed, and all three are elementary.
\end{remark}

\begin{definition}[Diagonal skeleton]
In the notation of Lemma~\ref{lem:small-injection}, a provider skeleton is
called \emph{diagonal} when
\[
 m_2b=m_1a.
\]
The shift equation then forces $m_2=m_1+h$, and
Lemma~\ref{lem:support} forces equal prime support.  Every other skeleton
is called off-diagonal.
\end{definition}

\begin{proposition}[Exact diagonal prime-pair parametrization]\label{prop:diag-param}
Let $h\ge1$ and $j\in\mathcal J_h$.  If $r\ge1$ satisfies the conditions in the definition of $D_{h,>y}^\varphi(x)$ except possibly the size and smoothness inequalities, then
\[
        n=j(A_jr+1)
\]
satisfies $\varphi(n)=\varphi(n+h)$.  Conversely, for every $h\ge1$, every diagonal skeleton in
Proposition~\ref{prop:transfer-diagonal} is obtained in this way, with
$j=m_1$ and $r$ equal to the common multiplier in the two provider primes.
\end{proposition}

\begin{proof}
Since $j\in\mathcal J_h$ we have $\rad(j)=\rad(j+h)$, so
$\varphi(n)/n=\prod_{p\mid n}(1-1/p)$ gives
\[
        \frac{\varphi(j)}{j}=\frac{\varphi(j+h)}{j+h}.
\]
Consequently
\[
        A_j\varphi(j)=\frac{j+h}{d_j}\varphi(j)
        =\frac{j}{d_j}\varphi(j+h)=B_j\varphi(j+h).
\]
Also
\[
        (j+h)(B_jr+1)-j(A_jr+1)=h,
\]
because $(j+h)B_j=jA_j=j(j+h)/d_j$.  Hence
\[
        n+h=(j+h)(B_jr+1).
\]
The primality and coprimality hypotheses give
\[
        \varphi(n)=\varphi(j)A_jr,\qquad
        \varphi(n+h)=\varphi(j+h)B_jr,
\]
and these are equal by the previous displayed identity.

Conversely, in a diagonal skeleton the notation of Lemma \ref{lem:small-injection} gives
\[
        m_2=m_1+h,\qquad
        \frac{\varphi(m_1)}{m_1}=\frac{\varphi(m_2)}{m_2}.
\]
Thus $m_1\in\mathcal J_h$.  If $d=(m_1,m_2)$, then the coprime pair $(a,b)$ occurring in the proof of Lemma \ref{lem:small-injection} is
\[
        a=\frac{m_2}{d},\qquad b=\frac{m_1}{d}.
\]
The two provider primes are $ar+1$ and $br+1$ for $r=cp$.  This is precisely the parametrization above.
\end{proof}

\begin{lemma}[Two fixed affine-linear forms]\label{lem:two-linear-sieve}
Let
\[
 L_1(r)=A_1r+B_1,\qquad L_2(r)=A_2r+B_2
\]
be primitive integer linear forms with positive leading coefficients,
positive constant terms, and nonzero determinant $A_1B_2-A_2B_1$.  Then,
for $R\ge3$,
\[
 \#\{1\le r\le R:L_1(r)\text{ and }L_2(r)\text{ are positive primes}\}
 \ll_{L_1,L_2}\frac{R}{(\log R)^2}.
\]
The implied constant depends on the two forms only through the set of
primes dividing $2\Delta$, where $\Delta=A_1A_2(A_1B_2-A_2B_1)$; in particular the estimate
is uniform over any family of pairs of forms whose leading coefficients and
determinants are supported on one fixed finite set of primes.
\end{lemma}

\begin{proof}
Put $F(r)=L_1(r)L_2(r)$ and, for a prime $\ell$, let
\[
 \nu(\ell)=\#\{a\pmod\ell:F(a)\equiv0\pmod\ell\}.
\]
If $\nu(\ell)=\ell$ for some prime $\ell$, then every value of $F$ is
divisible by $\ell$.  Whenever both linear values are prime, one of them
must equal $\ell$, which leaves at most two values of $r$.  We may
therefore assume $\nu(\ell)<\ell$ for every $\ell$.

Let
\[
 \Delta=A_1A_2(A_1B_2-A_2B_1).
\]
For $\ell\nmid\Delta$, each form has one root modulo $\ell$ and the two
roots are distinct, so $\nu(\ell)=2$.  For squarefree $d$, extend
$\nu$ multiplicatively.  The Chinese remainder theorem gives
\begin{equation}\label{eq:sieve-local-count}
 \#\{r\le R:d\mid F(r)\}
 =\frac{\nu(d)}dR+r_d,\qquad |r_d|\le\nu(d).
\end{equation}
The exceptional primes dividing $\Delta$ form a fixed finite set and,
because $\nu(\ell)<\ell$, contribute only a nonzero constant to the
sifting density.  Primitivity gives $\nu(\ell)\le2$ for every $\ell$ (if
$\ell\mid A_i$ then $\ell\nmid B_i$, so $L_i$ has no root modulo $\ell$),
and $\nu(\ell)=2$ for every $\ell\nmid\Delta$; hence the sieve hypotheses
$(\Omega_1)$ and $(\Omega_2(2,L))$ hold with absolute constants, and the
entire dependence of the implied constant on the two forms is through the
primes dividing $2\Delta$ (the prime $2$ is carried along so that it need not
be sifted separately; in the application below $\Delta$ is even, so this is
no loss).  This is the same upper-bound-sieve input used for the
analogous pair of linear forms in \cite[Theorem~3.1]{PPT}.

The Selberg upper-bound sieve in dimension two
\cite[Theorem~5.7]{HR}, applied to \eqref{eq:sieve-local-count}, gives
for $2\le z\le R^{1/4}$
\begin{equation}\label{eq:selberg-two-form}
 \#\{r\le R:(F(r),P(z))=1\}
 \ll_{L_1,L_2}
 R\prod_{\ell\le z}\left(1-\frac{\nu(\ell)}\ell\right)
 +\sum_{\substack{d\le z^2\\d\mid P(z)}}3^{\omega(d)}|r_d|.
\end{equation}
Since $\nu(\ell)=2$ outside a fixed set,
Mertens' theorem gives
\[
 \prod_{\ell\le z}\left(1-\frac{\nu(\ell)}\ell\right)
 \ll_{L_1,L_2}(\log z)^{-2}.
\]
Also $|r_d|\le2^{\omega(d)}$ up to a fixed multiplicative factor, so the
remainder in \eqref{eq:selberg-two-form} is
$O_{L_1,L_2}(z^2(\log z)^5)$.  Choose $z=R^{1/10}$.  A pair of prime
values both exceeding $z$ belongs to the sifted set; pairs for which one
prime value is at most $z$ contribute $O_{L_1,L_2}(z)$.  The displayed
bounds now give $O_{L_1,L_2}(R/(\log R)^2)$.
\end{proof}

\begin{proposition}[Sieve bound for the diagonal]\label{prop:D-bound}
For every fixed $h\ge1$ and every $y\ge2$,
\[
        D_{h,>y}^\varphi(x)\ll_h \frac{x}{(\log x)^2}
\]
for all sufficiently large $x$.
\end{proposition}

\begin{proof}
By Lemma~\ref{lem:J-finite}, for every $j\in\mathcal J_h$ the numbers
$A_j$, $B_j$ and $A_j-B_j$ have all their prime factors among those of
$h$, so the leading coefficients and the determinant of the pair
$L_1(r)=A_jr+1$, $L_2(r)=B_jr+1$ are supported on the primes dividing $h$.
The uniformity clause of Lemma~\ref{lem:two-linear-sieve} therefore
supplies a constant $c_h$, depending only on $\rad(h)$, with
\begin{equation}\label{eq:diag-uniform-sieve}
 \#\{r\le R:A_jr+1\text{ and }B_jr+1\text{ are prime}\}
 \le c_h\frac R{(\log R)^2}
 \qquad(R\ge3,\ j\in\mathcal J_h).
\end{equation}
Since $D_{h,>y}^\varphi(x)$ counts integers rather than representations,
it is bounded above by the number of representing pairs $(j,r)$, and the
conditions $P^+(\Gamma_jr)>y$ and the two coprimality conditions can only
reduce that count.

Split the range of $j$ at $x^{1/2}$.  If $j\le x^{1/2}$, then
$j(A_jr+1)\le x$ and $A_j\ge1$ force $r\le R_j:=x/j$, and
$R_j\ge x^{1/2}\ge3$ for $x\ge9$, so $\log R_j\ge\tfrac12\log x$ and
\eqref{eq:diag-uniform-sieve} gives
\[
 N_j:=\#\{r\le R_j:A_jr+1\text{ and }B_jr+1\text{ are prime}\}
 \le\frac{4c_h}{(\log x)^2}\cdot\frac xj.
\]
Every representing pair with $j\le x^{1/2}$ is counted by $N_j$, so by
\eqref{eq:J-h-quantitative}
\[
 \sum_{\substack{j\in\mathcal J_h\\ j\le x^{1/2}}}N_j
 \le\frac{4c_h}{(\log x)^2}\,x\sum_{j\in\mathcal J_h}\frac1j
 \le\frac{4c_h}{\varphi(h)}\cdot\frac{hx}{(\log x)^2}.
\]
If $j>x^{1/2}$, then $r\le x/j\le x^{1/2}$, so each such $j$ contributes at
most $x^{1/2}$ pairs, while by \eqref{eq:J-h-count} there are at most
$(1+\log x/\log2)^{\omega(h)}$ indices $j\in\mathcal J_h$ with $j\le x$.
That range therefore contributes at most
$x^{1/2}(1+\log x/\log2)^{\omega(h)}$, which is at most $x/(\log x)^2$
once $x\ge x_0(\omega(h))$.  Adding the two ranges proves the proposition.
\end{proof}

\begin{proposition}[Transfer with the diagonal displayed]\label{prop:transfer-diagonal}
Let $2\le y\le x$ and let $h\le y$ be arbitrary.  Every integer counted by $D_{h,>y}^\varphi(x)$ is counted by $B_h(x;y)$.  Define
\[
 \mathcal E_h(x;y)=B_h(x;y)-D_{h,>y}^\varphi(x).
\]
Thus, as an equality of cardinalities,
\begin{equation}\label{eq:diagonal-cardinality-identity}
 0\le B_h(x;y)-D_{h,>y}^\varphi(x)
 =\mathcal E_h(x;y).
\end{equation}
Then
\begin{align}\label{eq:transfer-diagonal}
 \mathcal E_h(x;y)
 &\ll
 \frac{x}{y}+x^{1/2}
 +\frac{x\log x}{y}
 +\frac{x(\log y)^2}{y} \notag\\
 &\quad
 +\sum_{y<p\le (2xy)^{1/2}}\Psi_\tau\left(\frac{2xy}{p^2},p\right),
\end{align}
with an absolute implied constant.  Equivalently, by \eqref{eq:diagonal-cardinality-identity},
\begin{align}
 B_h(x;y)
 &= D_{h,>y}^\varphi(x)
 +O\left(
 \frac{x}{y}+x^{1/2}
 +\frac{x\log x}{y}
 +\frac{x(\log y)^2}{y} \right. \notag\\
 &\qquad\qquad\left.
 +\sum_{y<p\le (2xy)^{1/2}}\Psi_\tau\left(\frac{2xy}{p^2},p\right)\right).
\end{align}
\end{proposition}

\begin{proof}
Partition the nonsmooth solutions into the square-source, nonsquare
small-cofactor, and nonsquare large-cofactor classes used in the proof of
Theorem~\ref{thm:transfer}.  The square-source calculation
\eqref{eq:square-transfer} and the large-cofactor calculations
\eqref{eq:sigma1-bound-transfer}--\eqref{eq:sigma2-bound-transfer} do not
use the parity of $h$; they therefore give the corresponding terms in
\eqref{eq:transfer-diagonal}.  In the small-cofactor class the equation
\[
        cp(m_2b-m_1a)=h+m_1-m_2
\]
shows that either $m_2b\ne m_1a$, in which case $C=cp$ is determined by $(m_1,m_2)$ and the formulas
$q_1=aC+1$, $q_2=bC+1$, $n=m_1q_1$ determine at most one $n$, or $m_2b=m_1a$.  In the latter case $m_2=m_1+h$ and Lemma \ref{lem:support} gives
$\rad(m_1)=\rad(m_2)$.  The provider primes in the skeleton occur to
exponent one and are coprime to their cofactors, so Proposition
\ref{prop:diag-param} represents the solution in $D_{h,>y}^\varphi(x)$ exactly.  There are at most
$\sum_{m_1m_2\le x/y}1\ll(x/y)\log(2x)$ such cofactor pairs,
which gives the stated small-cofactor term.  Combining the three classes
proves the proposition.
\end{proof}

\section{Large-prime-rank amplification}\label{sec:rank-amplification}

For $y\ge2$ define the large-prime rank
\[
 \Omega_y(m)=\sum_{p>y}\nu_p(m).
\]
The purpose of this section is to prove the fixed-rank theorem and then
track the same mechanism while the rank grows.  In the fixed-rank
subsections we take $T\asymp_JG$, $y=e^T$, and $Y=y^J$.  The supplier
system definitions and counting lemmas are, however, stated for every
integer $J\ge1$ with their dependence on $J$ explicit; only the analytic
specializations declare $J$ fixed.

\begin{lemma}[A reciprocal friable tail]\label{lem:reciprocal-friable-tail}
Let $R\ge3$, $H\ge3$, and put $u=\log H/\log R$.  If $u\to\infty$ and
$u\le R^{1/2}$, then
\begin{equation}\label{eq:reciprocal-friable-tail}
 \sum_{\substack{m>H\\P^+(m)\le R}}\frac1m
 \le \exp\left\{-u\log u+
 O\left(\log_2(3R)+\frac{u}{\log u}\right)\right\}.
\end{equation}
In particular, if $u\log u\gg\log_2(3R)$, the right side is at most
$\exp\{-\tfrac12u\log u\}$.
\end{lemma}

\begin{proof}
Put $\delta=\log u/\log R$.  Rankin's argument gives
\[
 \sum_{\substack{m>H\\P^+(m)\le R}}\frac1m
 \le H^{-\delta}\prod_{p\le R}(1-p^{-1+\delta})^{-1}.
\]
Using the prime-sum calculation in the proof of
Lemma~\ref{lem:weighted-smooth}, with the Euler-product exponent reduced
from two to one, gives
\[
 \sum_{p\le R}p^{-1+\delta}
 \ll \log_2(3R)+\frac{u}{\log u}.
\]
Since $\delta\log H=u\log u$, this proves the claim.
\end{proof}

For a finite multiset $\mathcal P$ of primes, write
$D(\mathcal P)$ for the product of its elements, with multiplicity.
A block below always means a nonempty submultiset.  For $j\ge1$, let
$\tau_j$ denote the ordered $j$-fold divisor function, so that
\[
 \sum_{m\ge1}\frac{\tau_j(m)}{m^s}=\zeta(s)^j
 \qquad(\Re s>1).
\]

\begin{lemma}[A growing-weight friable reciprocal tail]
\label{lem:growing-weight-tail}
Let $R,H\ge3$, let $j\ge1$, and put $v=\log H/\log R$.  Suppose
$v\ge3$ and $\log v\le\frac12\log R$.  Then
\begin{equation}\label{eq:growing-weight-tail}
 \sum_{\substack{m>H\\P^+(m)\le R}}\frac{\tau_j(m)}m
 \le
 \exp\left\{-v\log v+Cj\left(\log_2(3R)+\frac{v}{\log v}\right)\right\}
\end{equation}
for an absolute constant $C$.
\end{lemma}

\begin{proof}
Set $\delta=\log v/\log R$, so $0<\delta\le1/2$.  Rankin's method gives
\[
 \sum_{\substack{m>H\\P^+(m)\le R}}\frac{\tau_j(m)}m
 \le H^{-\delta}\prod_{p\le R}(1-p^{-1+\delta})^{-j}.
\]
The prime-sum estimate used in Lemma~\ref{lem:weighted-smooth} yields
\[
 \sum_{p\le R}-\log(1-p^{-1+\delta})
 \le C\left(\log_2(3R)+\frac{R^\delta}{\delta\log R}\right)
 =C\left(\log_2(3R)+\frac{v}{\log v}\right).
\]
Since $\delta\log H=v\log v$, the result follows.
\end{proof}

\begin{lemma}[Fixed-rank scale bookkeeping]\label{lem:rank-scales}
Let $J$ be fixed, $T\asymp_JG$, $y=e^T$, $Y=y^J$, and
$\epsilon=(\log_2 x)^{-1/2}$.  Then
\begin{gather*}
 V^{O_J(1)}(\log_2 x)^{O_J(1)}=e^{o_J(V)},\\
 Y^{J+1}<x^{\epsilon/(2J)},\\
 \frac{\log(x^{\epsilon/J})}{\log Y}
  \log\!\left(\frac{\log(x^{\epsilon/J})}{\log Y}\right)
  =\omega_J(G),\\
 \frac{\log(x^{1/(2J)})}{\log Y}
  \log\!\left(\frac{\log(x^{1/(2J)})}{\log Y}\right)
  =\omega_J(G).
\end{gather*}
\end{lemma}

\begin{proof}
The first assertion follows from $\log V,\log_3 x=o(V)$.  Since
$L/T\asymp_JV$ and $\epsilon V\to\infty$, the second follows from
$J(J+1)T=o(\epsilon L/J)$.  The two remaining expressions are respectively
of orders $\epsilon V\log V$ and $V\log V$ up to constants depending on
$J$; both dominate $G$ by Lemma~\ref{lem:scales} and the identity
$G/V=\sA\asymp\log_3 x$.
\end{proof}

\begin{definition}[Supplier systems]
\label{def:supplier-system}
Put $I=\{1,\ldots,J\}$, and let
\[
 \mathcal P=(p_i)_{i\in I},\qquad p_i\in(y,Y],
\]
be a labelled $J$-tuple of primes.  For $B\subseteq I$ put
\[
 D(B)=\prod_{i\in B}p_i,
 \qquad D(\mathcal P)=D(I).
\]
Let $B_J$ denote the $J$th Bell number, the number of set partitions of
$I$; in particular $B_J\le J^J$.

A $Y$-admissible supplier system $\mathfrak s$ for $\mathcal P$ consists
of a set partition $\pi$ of $I$ and, for every $B\in\pi$, a prime base
$q_B$, distinct for different blocks, together with exactly one of the
following descriptors.
\begin{enumerate}[label=\textup{(L)},leftmargin=2.2em]
\item $D(B)\mid q_B-1$, $P^+(q_B-1)\le Y$, $q_B\le2x$, and the selected
source divisor is $\delta_B=q_B$.
\item[\textup{(P)}] Put
\[
 E_B=\{i\in B:p_i=q_B\},\qquad e_B=|E_B|.
\]
Here $e_B\ge1$, $D(B\setminus E_B)\mid q_B-1$,
$P^+(q_B-1)\le Y$, $q_B\le Y$, and the selected source divisor is
$\delta_B=q_B^{e_B+1}$.
\end{enumerate}
The divisor selected by the system is
\[
 d(\mathfrak s)=\prod_{B\in\pi}\delta_B.
\]
Define
\begin{equation}\label{eq:formal-weight-definition}
 w_{\mathcal P}(d)=\#\{\mathfrak s:d(\mathfrak s)=d\},\qquad
 F_{\mathcal P}(N)=\sum_{d\mid N}w_{\mathcal P}(d),\qquad
 \Lambda(\mathcal P)=\sum_d\frac{w_{\mathcal P}(d)}d.
\end{equation}
For $H\ge1$ write
$\Lambda_{>H}(\mathcal P)=\sum_{d>H}w_{\mathcal P}(d)/d$.
\end{definition}

The definition is deliberately independent of an ambient integer.  A
\emph{realization} of $\mathfrak s$ in $N$ is the injective block map
\[
 \rho_{\mathfrak s}:\pi\longrightarrow
 \{q^a:q^a\Vert N\},\qquad B\longmapsto q_B^{a_B},
\]
for which $\delta_B\mid q_B^{a_B}$.  It induces the labelled occurrence
assignment
\[
 \sigma_{\mathfrak s}:I\longrightarrow
 \{q^a:q^a\Vert N\},\qquad \sigma_{\mathfrak s}(i)=\rho_{\mathfrak s}(B)
 \quad(i\in B).
\]
Distinct blocks use distinct local prime-power factors.  Because the bases
$q_B$ are distinct, a realization exists exactly when
$d(\mathfrak s)\mid N$.  Descriptor \textup{(L)} then supplies all labels
in $B$ through $q_B-1$, whereas \textup{(P)} supplies the labels in
$E_B$ through $q_B^{a_B-1}$ and the remaining labels through $q_B-1$.

\begin{lemma}[Reciprocal mass of one source block]
\label{lem:source-block-mass}
Let $B\subseteq I$ be nonempty and put $D_B=D(B)$.  For every integer
$J\ge1$, every $y\ge3$, and every $Y\le2x$, the reciprocal mass of all
admissible descriptors for $B$ is at most
\begin{equation}\label{eq:block-mass}
 \frac{\mathcal L_J(x,y)}{D_B},\qquad
 \mathcal L_J(x,y)=
 C\left((1+\log_2(3x))e^{2J/y}+J\right),
\end{equation}
where $C$ is absolute.  For a fixed exact local factor $q^a$ and a fixed
block $B$, at most one of the two descriptor types is possible.
\end{lemma}

\begin{proof}
For descriptors of type \textup{(L)}, Lemma~\ref{lem:reciprocal-BT}
gives
\[
 \sum_{\substack{q\le2x\\q\equiv1\pmod {D_B}\\q\ \mathrm{prime}}}\frac1q
 \ll \frac{1+\log_2(3x)}{\varphi(D_B)}.
\]
Every prime divisor of $D_B$ exceeds $y$.  Since there are at most $J$
distinct such primes and
$-\log(1-p^{-1})\le2/p$ for $p\ge3$,
\[
 \frac{D_B}{\varphi(D_B)}
 =\prod_{p\mid D_B}(1-p^{-1})^{-1}
 \le \exp(2J/y).
\]
This proves the linear part of \eqref{eq:block-mass}.

For a type-\textup{(P)} descriptor, the base $q$ must be one of the
prime values occurring among the labels in $B$.  Put
$e=\#\{i\in B:p_i=q\}$.  The divisibility condition gives
\[
 D_B=q^eD(B\setminus E_B)\le q^e(q-1)<q^{e+1},
\]
and hence $q^{-e-1}<D_B^{-1}$.  There are at most $J$ possible base
values $q$, so the total power-source mass is at most $J/D_B$.

Finally, for a fixed exact factor $q^a$, a block containing no demand
$p_i=q$ can only be linear.  If it contains such a demand, then
$q\nmid q-1$, so every label in the block with value $q$ must be supplied
by $q^{a-1}$; consequently $E_B$ is forced to be the set displayed in
Definition~\ref{def:supplier-system}.  Thus the local descriptor is
unique once $B$ and $q^a$ are fixed.
\end{proof}

\begin{lemma}[Low-rank compression]\label{lem:low-rank-compression}
Uniformly for $x\le X\le2x$,
\begin{equation}\label{eq:low-rank-compression}
 \#\{N\le X:P^+(\varphi(N))\le Y,
          \ \Omega_y(\varphi(N))<J\}
 \le X\exp\{-\mathcal S_x(T)+o_J(V)\},
\end{equation}
where
\[
 \mathcal S_x(T)=\frac{L}{T}
 \left(\log_2\frac LT+\log_3\frac LT\right).
\]
\end{lemma}

\begin{proof}
Set $\epsilon=(\log_2 x)^{-1/2}$.  For
$N=\prod q^a$, let $d$ be the product of those full prime powers $q^a\Vert N$
for which $\varphi(q^a)$ has a prime factor exceeding $y$.  There are fewer
than $J$ such local factors.  Writing $N=dm$, multiplicativity gives
\[
 P^+(\varphi(m))\le y.
\]

We first bound the reciprocal mass of the possible local factors in $d$.
For $a=1$, classify $q-1$ by the complete nonempty multiset
$\mathcal B$ of its prime factors in $(y,Y]$, counted with multiplicity.
Its size is below $J$.  For $1\le s<J$,
\[
 \sum_{\substack{\mathcal B:\,|\mathcal B|=s}}
 \frac1{D(\mathcal B)}
 \le \left(\sum_{y<p\le Y}\frac1p\right)^s
 \ll_J1
\]
by Mertens' theorem.  Lemma~\ref{lem:source-block-mass} contributes one
factor $O(1+\log_2(3x))$ for the selected linear local factor.  Summing
over the finitely many block sizes therefore gives total reciprocal mass
$\exp\{O_J(\log_3 x)\}=e^{o_J(V)}$ for one linear local factor.  If
$a\ge2$, then $q>y$, $q\le Y$, and $a\le J$, because
$q^{a-1}\mid\varphi(q^a)$ already supplies $a-1$ large-prime occurrences.
Hence
\[
 \sum_{2\le a\le J}\sum_{y<q\le Y}\frac1{q^a}=O_J(1).
\]
Taking products of fewer than $J$ local factors preserves the bound
$e^{o_J(V)}$.  It follows that
\begin{equation}\label{eq:low-rank-d-mass}
 \sum_d\frac1d\le e^{o_J(V)},
\end{equation}
where the sum ranges over all possible products of fewer than $J$ local
factors.

If $d\le x^\epsilon$, then $X/d\ge x^{1-\epsilon}$.  Lemma~\ref{lem:bfps-uniform}\textup{(i)} gives
\[
 \Phi(X/d,y)\le \frac{X}{d}
 \exp\{-\mathcal S_x(T)+o_J(V)\},
\]
uniformly in this range; the displayed $o_J(V)$ is the error function in
\eqref{eq:bfps-fixed-uniform} together with the explicit replacement of
$\log(X/d)$ by $L$.  Summing with \eqref{eq:low-rank-d-mass} handles all
such $d$.

If $d>x^\epsilon$, one of its fewer than $J$ local factors exceeds
$x^{\epsilon/J}$.  A factor with exponent $a\ge2$ is at most
$Y^J=e^{J^2T}<x^{\epsilon/J}$ for large $x$, by
Lemma~\ref{lem:rank-scales}.  Hence the large factor is a
prime $q$ with $P^+(q-1)\le Y$.  Dropping primality and writing $q-1=r$,
Lemma~\ref{lem:reciprocal-friable-tail} gives
\[
 \sum_{\substack{q>x^{\epsilon/J}\\P^+(q-1)\le Y}}\frac1q
 \le 2\sum_{\substack{r>x^{\epsilon/J}/2\\P^+(r)\le Y}}\frac1r
 \le \exp\{-\omega_J(G)\}.
\]
Indeed the relevant smoothness parameter is
$\asymp_J\epsilon L/T$, and Lemma~\ref{lem:rank-scales} shows that its
Rankin exponent is $\omega_J(G)$.  The other local factors have reciprocal
mass $e^{o_J(V)}$.  Counting multiples of $d$ therefore gives a total
$O(Xe^{-\omega_J(G)})$, which is negligible.  This proves the lemma.
\end{proof}

\begin{lemma}[Occurrence assignment and coverage]
\label{lem:supplier-coverage}
Let $N\le2x$, suppose $P^+(\varphi(N))\le Y$, and let
$\mathcal P=(p_i)_{i\in I}$ be a labelled $J$-tuple in $(y,Y]$ satisfying
$D(\mathcal P)\mid\varphi(N)$.  Then a $Y$-admissible supplier system for
$\mathcal P$ has selected divisor dividing $N$.  In particular,
\begin{equation}\label{eq:supplier-coverage}
 D(\mathcal P)\mid\varphi(N)\quad\Longrightarrow\quad
 F_{\mathcal P}(N)\ge1.
\end{equation}
Every prime factor of every selected divisor exceeds $y$.
\end{lemma}

\begin{proof}
Write
\[
 N=\prod_{q^a\Vert N}q^a,
 \qquad
 \varphi(N)=\prod_{q^a\Vert N}q^{a-1}(q-1).
\]
For each prime $p$, list all occurrences of $p$ in the displayed product,
first by increasing base $q$, then with the occurrences in $q^{a-1}$
before those in $q-1$, and finally by multiplicity.  Assign the labels
$i$ with $p_i=p$, in increasing label order, to the first available
$p$-slots.  The divisibility assumption guarantees that this assignment
is possible.

Partition $I$ according to the exact local factor $q^a$ receiving a
label.  Let $B$ be one fibre.  If no label is assigned to $q^{a-1}$, then
$D(B)\mid q-1$ and descriptor \textup{(L)} applies.  Otherwise, precisely
the labels in
$E_B=\{i\in B:p_i=q\}$ are assigned to $q^{a-1}$: no occurrence of $q$
can occur in $q-1$.  Thus $e_B=|E_B|\le a-1$,
$D(B\setminus E_B)\mid q-1$, and
$q^{e_B+1}\mid q^a$, so descriptor \textup{(P)} applies.  Distinct fibres
come from distinct exact prime-power factors and therefore have distinct
bases.  Since each local totient divides the product defining
$\varphi(N)$, all required smoothness conditions hold.  Multiplying the
selected divisors gives a divisor of $N$, proving
\eqref{eq:supplier-coverage}.

For type \textup{(P)}, the base is one of the demanded primes and exceeds
$y$.  For type \textup{(L)}, $q_B>D(B)\ge p_i>y$ for every $i\in B$.
This proves the support assertion.
\end{proof}

\begin{remark}[A repeated-prime power-source example]
\label{rem:supplier-example}
Take labelled demands
$\mathcal P=(7_1,7_2,3_3,5_4)$.  The block
$B_1=\{1,2,3\}$ may use the type-\textup{(P)} source $7^3$, because
\[
 \varphi(7^3)=7^2\cdot6
\]
supplies the two labelled copies of $7$ through $7^2$ and the labelled
copy of $3$ through $6$.  Here $E_{B_1}=\{1,2\}$ is forced, not chosen.
The block $B_2=\{4\}$ may use the type-\textup{(L)} source $11$, since
$5\mid10$.  The resulting selected divisor is $7^3\cdot11$.  This
example also shows why repeated demands must remain labelled even though
the numerical prime values may coincide.
\end{remark}

\begin{lemma}[Block reciprocal mass and pointwise multiplicity]
\label{lem:multiset-supplier-mass}
For every integer $J\ge1$, every $y\ge3$, every $Y\le2x$, and every
labelled $J$-tuple $\mathcal P$ in $(y,Y]$, one has
\begin{equation}\label{eq:multiset-mass}
 \Lambda(\mathcal P)
 \le \frac{B_J\mathcal L_J(x,y)^J}{D(\mathcal P)}.
\end{equation}
If
\[
 W_y(N)=\#\{q^a\Vert N:P^+(\varphi(q^a))>y\},
\]
then, for $N\le2x$,
\begin{equation}\label{eq:formal-pointwise}
 F_{\mathcal P}(N)
 \le B_J\max(1,W_y(N))^J,
 \qquad
 W_y(N)\le\frac{\log(2x)}{T}.
\end{equation}
\end{lemma}

\begin{proof}
Fix a partition $\pi$.  Dropping the distinct-base restriction and using
Lemma~\ref{lem:source-block-mass} independently on its blocks gives
\[
 \sum_{\mathfrak s:\,\pi(\mathfrak s)=\pi}\frac1{d(\mathfrak s)}
 \le \prod_{B\in\pi}\frac{\mathcal L_J(x,y)}{D(B)}
 \le\frac{\mathcal L_J(x,y)^J}{D(\mathcal P)}.
\]
Summation over the $B_J$ partitions proves \eqref{eq:multiset-mass}.

A system counted by $F_{\mathcal P}(N)$ assigns every nonempty block to a
distinct exact local factor counted by $W_y(N)$.  By the final assertion
of Lemma~\ref{lem:source-block-mass}, once the block and exact factor are
fixed there is at most one descriptor.  A partition with $b$ blocks has
at most $W_y(N)^b$ realizations, and summation over partitions proves the
first estimate in \eqref{eq:formal-pointwise}.

Every exact factor counted by $W_y(N)$ contributes at least one prime
factor greater than $y=e^T$ to $\varphi(N)$.  Their local totients divide
$\varphi(N)\le N\le2x$, so $y^{W_y(N)}\le2x$, which proves the second
estimate.
\end{proof}

For $H\ge1$ let
\begin{equation}\label{eq:large-system-set}
 \mathscr S_J(H)=
 \{(\mathcal P,\mathfrak s):\mathcal P\in((y,Y]\cap\mathbb P)^J,
      \ \mathfrak s\text{ is $Y$-admissible for }\mathcal P,
      \ d(\mathfrak s)>H\}.
\end{equation}
Because $w_{\mathcal P}(d)$ counts supplier systems rather than numerical
values alone, one has the exact identity
\begin{equation}\label{eq:aggregate-tail-identity}
 \sum_{\mathcal P}\Lambda_{>H}(\mathcal P)
 =\sum_{(\mathcal P,\mathfrak s)\in\mathscr S_J(H)}
      \frac1{d(\mathfrak s)}.
\end{equation}
Repeated numerical primes in $\mathcal P$ remain distinguishable through
the labels in $I$.

\begin{lemma}[Injective encoding and reconstruction]
\label{lem:supplier-encoding-injective}
Put
\[
 M_J(x,T)=1+J+\left\lfloor\frac{\log(2x)}T\right\rfloor.
\]
Order the sources of a supplier system by increasing base.  A linear
source is represented by $r=q-1$.  A power source is represented by the
pair $(q,e)$, where $e=|E_B|$.  For an integer $r$, let
$\mathcal O_y(r)$ be the ordered list of prime-factor occurrences of $r$
which exceed $y$, ordered first by prime value and then by occurrence
number.  Give a linear source the slot list $\mathcal O_y(q-1)$.  Give a
power source $e$ formal base slots carrying $q$, followed by
$\mathcal O_y(q-1)$.  The codomain is the disjoint union, over
$1\le b\le J$ and $\mathbf t\in\{L,P\}^b$, of tuples consisting of the
type data, a numerical entry $(q_s,e_s)$ at each $P$-position, a numerical
entry $r_s$ at each $L$-position, and a prime-preserving map from $I$
into the disjoint union of the resulting slot lists, the sources being
listed in a canonical order determined by the supplier system.  Only two
features of this slot family are used below: that it is canonical, and
that its total cardinality is at most $JM_J(x,T)$; any indexing with those
two properties may be substituted.

The encoding of $(\mathcal P,\mathfrak s)$ is
\begin{equation}\label{eq:supplier-encoding-datum}
 \operatorname{Enc}(\mathcal P,\mathfrak s)=
 \left(b,\mathbf t,
       (q_s,e_s)_{t_s=P},
       (r_s)_{t_s=L},
       \sigma\right),
\end{equation}
where $b$ is the number of sources, $\mathbf t\in\{L,P\}^b$ is the
source-type word in base order, the two numerical tuples occupy the
positions prescribed by $\mathbf t$, and
$\sigma:I\to\bigsqcup_{s=1}^b\mathcal S_s$ records the slot occupied by
each labelled demand.  For each source and each prime value, match the relevant labels in increasing
label order to the matching slots in occurrence order; this makes $\sigma$ a
canonical function of the supplier system, and the prime carried by
$\sigma(i)$ is $p_i$ for every $i\in I$.
Then $\operatorname{Enc}$ is injective.

For fixed numerical power and linear tuples, the number of possible
values of all remaining combinatorial data in
\eqref{eq:supplier-encoding-datum} is at most
\begin{equation}\label{eq:supplier-slot-count}
 [CJM_J(x,T)]^{3J}
 \le[CJ(1+J+\log(2x)/T)]^{3J}
\end{equation}
for an absolute constant $C$.
\end{lemma}

\begin{proof}
We give the decoding algorithm.  The type word places every numerical
entry in a unique source position.  For a linear entry $r_s$, recover the
base as $q_s=r_s+1$ and construct the complete slot list
$\mathcal O_y(r_s)$.  For a power entry $(q_s,e_s)$, construct the
$e_s$ formal $q_s$-slots and then factor $q_s-1$ to obtain the remainder
of its slot list.  The map $\sigma$ identifies, by label, exactly which
slots are used.  Labels mapped to source $s$ recover its block $B_s$;
labels in the formal base slots recover $E_{B_s}$, and the other labels
recover the demanded occurrences in $q_s-1$.  Thus the labelled tuple,
its partition, the ordered distinct bases, every descriptor, and the
selected divisor are all reconstructed.  Equal numerical demands cause
no ambiguity because their labels and occurrence slots are distinct.
This proves injectivity.

It remains to justify \eqref{eq:supplier-slot-count}.  There are at most
$J$ sources.  The number of choices for $b$, the type word, and the placement of the
two numerical subtuples is at most $(CJ)^J$.  The power exponents
$1\le e_s\le J$, together with ordering and padding conventions, contribute
at most a second factor $(CJ)^J$.  A linear source contains at most
$\log(2x)/T$ slots, while a power source contains at most $J$ formal
slots plus that many factorization slots.  Hence the disjoint union of
all source slots has cardinality at most
\[
 J\left(J+\frac{\log(2x)}T\right)\le J M_J(x,T).
\]
Allowing every map of the $J$ labelled demands into this enlarged slot
set has cardinality at most $[JM_J(x,T)]^J$.  Multiplication of the three displayed
bounds proves \eqref{eq:supplier-slot-count}.  The count deliberately
overcounts invalid maps, collisions of source bases, and maps which do
not preserve prime values; each enlargement preserves the asserted upper bound.
\end{proof}

\begin{proposition}[Aggregate large-source tail]
\label{prop:supplier-aggregate-tail}
For every integer $J\ge1$, every $y\ge3$, every $Y\le2x$, and every
$H\ge3$,
\begin{equation}\label{eq:aggregate-tail}
 \sum_{\mathcal P}\Lambda_{>H}(\mathcal P)
 \le \mathcal E_J(x;y,T)
 \sum_{\substack{m>2^{-J}HY^{-2J}\\P^+(m)\le Y}}
       \frac{\tau_J(m)}m,
\end{equation}
where the first sum is over all labelled $J$-tuples of primes in $(y,Y]$
and
\begin{equation}\label{eq:formal-encoding-factor}
 \mathcal E_J(x;y,T)
 \le [CJ(1+J+\log(2x)/T)]^{3J}\exp(CJ/y).
\end{equation}
\end{proposition}

\begin{proof}
Start from the exact identity \eqref{eq:aggregate-tail-identity} and
apply the injection of Lemma~\ref{lem:supplier-encoding-injective}.  Let
$d_{\rm pow}$ be the product of the selected divisors belonging to power
sources.  If their selected exponents are $e_s+1$, then
$\sum_s e_s\le J$ and the number of power sources is at most $J$.
Since every power base is at most $Y$,
\begin{equation}\label{eq:power-source-product}
 d_{\rm pow}\le Y^{\sum_s(e_s+1)}\le Y^{2J}.
\end{equation}
After the label, distinctness, and divisibility restrictions are dropped,
the reciprocal mass of an ordered power-source tuple is bounded by
\[
 \left(1+\sum_{q>y}\sum_{e\ge1}q^{-e-1}\right)^J
 \le\exp(CJ/y).
\]
This accounts for the exponential factor in
\eqref{eq:formal-encoding-factor}; the choices of the exponents $e_s$
were already included in the combinatorial factor
\eqref{eq:supplier-slot-count}.

Write the linear source primes in source order as
$q_1,\ldots,q_t$ and set $r_i=q_i-1$.  Admissibility gives
$P^+(r_i)\le Y$.  If the complete selected divisor is greater than $H$,
then \eqref{eq:power-source-product} implies
\[
 \prod_{i=1}^tq_i>HY^{-2J}.
\]
Since $q_i=r_i+1\le2r_i$ and $t\le J$,
\begin{equation}\label{eq:linear-product-threshold}
 m:=\prod_{i=1}^tr_i>2^{-J}HY^{-2J}.
\end{equation}
Moreover $q_i^{-1}\le r_i^{-1}$.  Pad the ordered tuple
$(r_1,\ldots,r_t)$ by ones to length $J$.  For each fixed product $m$,
the number of resulting ordered $J$-tuples is at most $\tau_J(m)$.
Thus summing the numerical linear data gives the weighted friable tail in
\eqref{eq:aggregate-tail}.  Multiplying by the power-source mass and the
three explicitly counted combinatorial factors proves
\eqref{eq:aggregate-tail}--\eqref{eq:formal-encoding-factor}.
\end{proof}

\begin{proposition}[High-rank shifted correlation]
\label{prop:high-rank-correlation}
Uniformly for $1\le h\le y$,
\begin{multline}\label{eq:high-rank-correlation}
 \#\{n\le x:\ \varphi(n)=\varphi(n+h)=M,\ P^+(M)\le Y,\\
                  \Omega_y(M)\ge J\}
 \ll_J x\exp\{-JT+o_J(V)\}.
\end{multline}
\end{proposition}

\begin{proof}
For each common value $M$, choose the first $J$ prime occurrences above
$y$ in nondecreasing prime order, breaking equal-prime ties by their
multiplicity index.  This gives a labelled tuple $\mathcal P$.  By
Lemma~\ref{lem:supplier-coverage}, on each shifted integer the indicator
of $D(\mathcal P)\mid\varphi(N)$ is at most $F_{\mathcal P}(N)$.
Consequently the contribution attached to $\mathcal P$ is at most
\[
 \sum_{n\le x}F_{\mathcal P}(n)F_{\mathcal P}(n+h).
\]
Every prime factor of the weight support exceeds $y\ge h$, so
Lemma~\ref{lem:shifted-divisor-convolution} applies.  It gives
\begin{equation}\label{eq:fixed-shifted-convolution-use}
 \sum_{n\le x}F_{\mathcal P}(n)F_{\mathcal P}(n+h)
 \le2x\Lambda(\mathcal P)^2
 +4x\|F_{\mathcal P}\|_{\infty,2x}
       \Lambda_{>\sqrt x}(\mathcal P).
\end{equation}

For fixed $J$, Lemma~\ref{lem:multiset-supplier-mass} gives
\[
 \Lambda(\mathcal P)
 \le\frac{e^{o_J(V)}}{D(\mathcal P)},
 \qquad
 \|F_{\mathcal P}\|_{\infty,2x}\le e^{o_J(V)}.
\]
Indeed $\log\mathcal L_J(x,y)=O_J(\log_3 x)$ and
$W_y(N)\ll_JV$.  Summing the first term of
\eqref{eq:fixed-shifted-convolution-use} over all labelled tuples gives
\begin{align*}
 2x\sum_{\mathcal P}\Lambda(\mathcal P)^2
 &\le xe^{o_J(V)}
 \left(\sum_{p>y}\frac1{p^2}\right)^J\\
 &\ll_J x\exp\{-JT+o_J(V)\}.
\end{align*}

For the second term, Proposition~\ref{prop:supplier-aggregate-tail}, with
$H=\sqrt x$, gives
\[
 \sum_{\mathcal P}\Lambda_{>\sqrt x}(\mathcal P)
 \le e^{o_J(V)}
 \sum_{\substack{m>H_0\\P^+(m)\le Y}}\frac{\tau_J(m)}m,
 \qquad
 H_0=2^{-J}x^{1/2}Y^{-2J}.
\]
Here $\log H_0=\frac12L-o_J(L)$ and
$\log Y=JT\asymp_JG$.  Lemma~\ref{lem:growing-weight-tail} therefore
makes the last sum $e^{-\omega_J(G)}$.  After multiplication by the
pointwise factor in \eqref{eq:fixed-shifted-convolution-use}, the
large-divisor contribution is negligible.  This proves the proposition.
\end{proof}

\begin{theorem}[Fixed-rank amplification with the diagonal displayed]
\label{thm:rank-amplification-variable}
Fix $J\ge1$ and suppose $T\asymp_JG$, $y=e^T$, $Y=y^J$, and $h\le y$.
Then
\begin{equation}\label{eq:rank-amplification-variable}
 S_h^\varphi(x)=D_{h,>Y}^\varphi(x)
 +O_J\!\left(x\exp\{-\min(\mathcal S_x(T),JT)+o_J(V)\}\right).
\end{equation}
\end{theorem}

\begin{proof}
Let $M=\varphi(n)=\varphi(n+h)$.  If $P^+(M)>Y$, Proposition
\ref{prop:transfer-diagonal}, applied with cutoff $Y$, gives the exact
inclusion and cardinality relation
\[
 0\le
 \#\{n\le x:P^+(M)>Y\}-D_{h,>Y}^\varphi(x)
 \le \mathcal E_h(x;Y).
\]
Proposition~\ref{prop:b4-general}, with
$\log Y=JT\asymp_JG$, then yields
\[
 0\le
 \#\{n\le x:P^+(M)>Y\}-D_{h,>Y}^\varphi(x)
 \ll_J xe^{-JT+o_J(V)}.
\]
All elementary transfer terms are absorbed here, and $x^{1/2}$ is
negligible.

Suppose now that $P^+(M)\le Y$.  If $\Omega_y(M)<J$, the integer $n$ is
counted by Lemma~\ref{lem:low-rank-compression}.  If
$\Omega_y(M)\ge J$, Proposition~\ref{prop:high-rank-correlation} applies.
Combining the three disjoint cases proves \eqref{eq:rank-amplification-variable}.
\end{proof}

\begin{theorem}[All shifts: amplified diagonal decomposition]
\label{thm:all-shifts}
Fix $J\ge1$ and put
\[
 T_J=G/\sqrt J,\qquad y_J=e^{T_J},\qquad Y_J=e^{\sqrt JG}.
\]
There is a function $\varepsilon_J(x)\to0$ such that, for
$x\ge x_0(J)$, uniformly for positive integers $h\le y_J$,
\[
 S_h^\varphi(x)=D_{h,>Y_J}^\varphi(x)
 +O_J\!\left(xe^{-\sqrt JG+\varepsilon_J(x)V}\right).
\]
For every fixed $h$,
\[
 S_h^\varphi(x)\ll_{h,J}\frac{x}{(\log x)^2}
 +xe^{-\sqrt JG+\varepsilon_J(x)V}.
\]
If $h$ is odd, the diagonal is empty.
\end{theorem}

\begin{proof}
Apply Theorem~\ref{thm:rank-amplification-variable} with
$T=T_J$.  Since
\[
 \mathcal S_x(T_J)=\sqrt JG+o_J(V),
 \qquad JT_J=\sqrt JG,
\]
the asserted decomposition follows.  Proposition~\ref{prop:D-bound}
bounds the diagonal for fixed $h$, and Lemma~\ref{lem:J-finite} makes it
empty for odd $h$.
\end{proof}

\begin{proof}[Proof of Theorem~\ref{thm:main}]
Apply Theorem~\ref{thm:all-shifts}; Lemma~\ref{lem:J-finite} removes the diagonal for odd shifts.
\end{proof}

\subsection{Moving rank on the \texorpdfstring{$\sqrt{\log x\log_2 x}$}{sqrt(log x loglog x)} scale}
\label{subsec:moving-rank}

The fixed-rank proof is uniform enough to permit the rank to grow, but its
large-supplier estimates must then be replaced by a divisor-weighted
friable tail.  We give all dependence on the moving rank explicitly.
Put
\[
 L_2=\log_2 x,\qquad L_3=\log_3 x,\qquad
 \mathcal R=(LL_2)^{1/2},\qquad \eta=\sA^{-1/2}.
\]
For all sufficiently large $x$, define
\begin{equation}\label{eq:moving-parameters}
 E=(\tfrac12-\eta)\mathcal R,
 \qquad K_0=(\tfrac12-2\eta)\mathcal R,
 \qquad T=\frac{L\sA}{2E},
\end{equation}
\begin{equation}\label{eq:moving-JK}
 J=\left\lfloor\frac{K_0}{T}\right\rfloor,
 \qquad K=JT,
 \qquad y=e^T,
 \qquad Z=e^K.
\end{equation}
Notice that $J$ is no longer fixed.  The quantities are defined in the
displayed order; no auxiliary parameter is subsequently sent to zero.

\begin{lemma}[Moving-rank bookkeeping]\label{lem:moving-bookkeeping}
With the parameters in \eqref{eq:moving-parameters}--\eqref{eq:moving-JK},
\begin{align}
 T&=(1+o(1))\sA\sqrt{\frac{L}{L_2}},
 &J&=(\tfrac12+o(1))\frac{L_2}{\sA},
 \label{eq:moving-asymptotics}\\
 K&=(\tfrac12-2\eta+o(\eta))\mathcal R.
 \label{eq:moving-K}
\end{align}
Moreover,
\begin{equation}\label{eq:moving-overheads}
 T+J^2L_2=o(\eta\mathcal R),
 \qquad
 JK=o(\eta L).
\end{equation}  Finally, if $3\le H\le2x$ and a function $\delta_H(x)\to0$ satisfies
\begin{equation}\label{eq:moving-H}
 \log H\ge\frac12L-\delta_H(x)\eta L,
 \qquad v=\frac{\log H}{K},
\end{equation}
then
\begin{equation}\label{eq:moving-tail-margin}
 v\log v\ge K+3\eta\mathcal R
\end{equation}
and
\begin{equation}\label{eq:moving-tail-error}
 J\left(\log_2(3Z)+\frac{v}{\log v}\right)
 =o(\eta\mathcal R)
\end{equation}
for all sufficiently large $x$.
\end{lemma}

\begin{proof}
Since $\sA\sim L_3$ and $\eta=\sA^{-1/2}$, one has
$\eta\to0$ and $\sA/L_2=o(\eta)$.  Direct substitution gives
\[
 T=\frac{\sA}{1-2\eta}\sqrt{\frac{L}{L_2}},
 \qquad
 J=\left(\frac12+o(1)\right)\frac{L_2}{\sA}.
\]
Moreover,
\[
 \frac{T}{\eta\mathcal R}
 \asymp\frac{\sA^{3/2}}{L_2}\longrightarrow0,
 \qquad
 \frac{J^2L_2}{\eta\mathcal R}
 \ll\frac{L_2^{5/2}}{\sA^{3/2}\sqrt L}\longrightarrow0,
\]
and
\[
 \frac{JK}{\eta L}
 \ll\frac{L_2^{3/2}}{\sqrt{\sA L}}\longrightarrow0.
\]
These estimates prove \eqref{eq:moving-overheads}.  Since
$0\le K_0-K<T=o(\eta\mathcal R)$, they also give
\eqref{eq:moving-K}.

Write $c=K/\mathcal R$.  By \eqref{eq:moving-K},
$c=\frac12-2\eta+o(\eta)$.  The smallest admissible $v$ occurs at the lower endpoint in
\eqref{eq:moving-H}; there
\[
 v=\left(\frac12-o(\eta)\right)\frac{L}{K}
   =\left(\frac1{2c}-o(\eta)\right)\sqrt{\frac{L}{L_2}},
\]
and
\[
 \log v=\frac12L_2-\frac12L_3+O(1)+o(\eta L_2).
\]
Consequently
\begin{align*}
 v\log v
 &=\left(\frac1{4c}+o(\eta)\right)\mathcal R\\
 &=\left(\frac12+2\eta+O(\eta^2)+o(\eta)\right)\mathcal R.
\end{align*}
Comparison with \eqref{eq:moving-K} proves
\eqref{eq:moving-tail-margin}.  The upper bound $H\le2x$ also gives
$v\ll\sqrt{L/L_2}$.  Consequently
\[
 J\log_2(3Z)\ll \frac{L_2^2}{\sA},
 \qquad
 \frac{Jv}{\log v}\ll\frac1{\sA}\sqrt{\frac{L}{L_2}},
\]
and both quantities are $o(\eta\mathcal R)$.  This proves
\eqref{eq:moving-tail-error}.
\end{proof}
\begin{lemma}[Admissibility of every moving-rank range]
\label{lem:moving-admissibility}
There is an absolute $x_0$ such that, for $x\ge x_0$, the parameters in
\eqref{eq:moving-parameters}--\eqref{eq:moving-JK} satisfy
\begin{equation}\label{eq:moving-domain-basic}
 \sA>0,\quad J\ge2,\quad 3\le y<Z<x,\quad 3Z^3<2x.
\end{equation}
Consequently every shift in the theorem range satisfies
$1\le h\le y<Z$.  Put
\[
 H_{\rm src}=2^{-J}x^{1/2}Z^{-2J},\qquad
 v_{\rm src}=\frac{\log H_{\rm src}}K,
\]
and
\[
 u_{\rm tr}^{-}=\frac{L-3K+\log 2}{2K},\qquad
 u_{\rm tr}^{+}=\frac{L-K+\log 2}{K}.
\]
Then
\begin{equation}\label{eq:moving-domain-tail}
 H_{\rm src}>3,\qquad v_{\rm src}\ge3,
 \qquad \log v_{\rm src}\le\tfrac12\log Z,
\end{equation}
and
\begin{equation}\label{eq:moving-domain-transfer}
 u_{\rm tr}^{-}\ge u_0,
 \qquad \log u_{\rm tr}^{+}\le\tfrac12\log Z,
 \qquad u_{\rm tr}^{+}\le Z^{1/2}.
\end{equation}
In particular, every application of
Lemmas~\ref{lem:weighted-smooth} and \ref{lem:growing-weight-tail} in the
moving-rank proof lies in its stated domain, including the endpoints
$W=x^{1/2}$, $W=2x$, $p=Z$, and $p=Z^2$.
\end{lemma}

\begin{proof}
Lemma~\ref{lem:moving-bookkeeping} gives
\[
 T\to\infty,\qquad J\to\infty,\qquad K\sim\tfrac12\mathcal R=o(L),
 \qquad T=o(K).
\]
Since $3K=o(L)$, they imply $3Z^3<2x$ as well as all the other
claims in \eqref{eq:moving-domain-basic}.  They also give
\[
 \log H_{\rm src}
 =\tfrac12L-2JK-J\log 2
 =\tfrac12L-o(\eta L),
\]
because $JK=o(\eta L)$.  Hence $H_{\rm src}>3$ and
\[
 v_{\rm src}\asymp\frac{L}{K}\asymp\sqrt{\frac{L}{L_2}}\to\infty.
\]
Thus $\log v_{\rm src}=O(L_2)$, whereas
$\tfrac12\log Z=K/2\asymp\mathcal R$, proving
\eqref{eq:moving-domain-tail}.

The same calculation, with
$\tfrac12(L-3K+\log 2)=\tfrac12L-o(\eta L)$, gives
$u_{\rm tr}^{-}\asymp\sqrt{L/L_2}\to\infty$, and hence
$u_{\rm tr}^{-}\ge u_0$ for all sufficiently large $x$.  Also
$u_{\rm tr}^{+}\asymp\sqrt{L/L_2}$, so
$\log u_{\rm tr}^{+}=O(L_2)=o(K)$ and
$u_{\rm tr}^{+}\le e^{K/2}=Z^{1/2}$.  This proves
\eqref{eq:moving-domain-transfer}.  Finally, the BFPS endpoint conditions
were checked explicitly in Lemma~\ref{lem:bfps-uniform}; the remaining
friable-tail hypotheses follow from the displayed inequalities and from
$v\log v\gg K\gg\log_2(3Z)$.
\end{proof}

\begin{proposition}[Moving-rank dependency ledger]
\label{lem:moving-dependence-ledger}
Let $\mathcal E_J(x;y,T)$ be the encoding factor in
\eqref{eq:formal-encoding-factor}, with $Y=Z$.  Every constant used in
the supplier definition, block mass, encoding, shifted convolution, and
growing-weight tail is absolute.  The fixed-rank statements containing
an unspecified $O_J$ or $o_J$ are not invoked in the moving-rank proof.
The complete losses used at moving rank are bounded as follows.
\begin{center}
\small
\renewcommand{\arraystretch}{1.14}
\begin{tabularx}{\textwidth}{@{}>{\raggedright\arraybackslash}p{0.27\textwidth}>{\raggedright\arraybackslash}p{0.27\textwidth}>{\raggedright\arraybackslash}X@{}}
\toprule
loss & source & upper order and normalization \\
\midrule
$T$ & flooring $K_0/T$ &
$\sA\sqrt{L/L_2}=o(\eta\mathcal R)$ \\
$\log B_J$ & labelled set partitions &
$J\log J=O(L_2)=o(\eta\mathcal R)$ \\
$J\log\mathcal L_J(x,y)$ & block reciprocal masses &
$O(JL_3)=O(L_2)=o(\eta\mathcal R)$ \\
$J\log(1+L/T)$ & pointwise source assignment &
$O(L_2^2/\sA)=o(\eta\mathcal R)$ \\
$J^2$ and $J^2\log_2(3J)$ & repeated labels and power-source data &
$O(L_2^2\log_4 x/\sA^2)=o(\eta\mathcal R)$ \\
$3J\log[CJ(1+J+L/T)]$ & type word, source order, padding and slot map &
$O(L_2^2/\sA)=o(\eta\mathcal R)$ \\
$CJ/y$ & power-source reciprocal mass &
$o(1)=o(\eta\mathcal R)$ \\
$\log\mathcal E_J(x;y,T)$ & complete aggregate encoding &
$O(L_2^2/\sA)=o(\eta\mathcal R)$ \\
$J\{\log_2(3Z)+v/\log v\}$ & $\tau_J$ Euler-product loss &
$o(\eta\mathcal R)$ by \eqref{eq:moving-tail-error} \\
$JK$ & source threshold $Z^{2J}$ &
$o(\eta L)$ \\
\bottomrule
\end{tabularx}
\end{center}
In particular, if
\begin{align}\label{eq:dependence-ledger-definition}
 \epsilon_{\rm dep}(x)
 ={}&\frac{T+\log B_J+J\log\mathcal L_J(x,y)
       +J\log(1+L/T)}{\eta\mathcal R}\notag\\
 &+\frac{J^2+J^2\log_2(3J)
       +3J\log[CJ(1+J+L/T)]
       +CJ/y}{\eta\mathcal R}\notag\\
 &+\frac{J}{\eta\mathcal R}
       \left(\log_2(3Z)+\frac{v_{\rm src}}{\log v_{\rm src}}\right)
 +\frac{JK}{\eta L},
\end{align}
then $\epsilon_{\rm dep}(x)\to0$, and
\[
 \log\mathcal E_J(x;y,T)
 \le \epsilon_{\rm dep}(x)\eta\mathcal R.
\]
\end{proposition}

\begin{proof}
Use $J\asymp L_2/\sA$, $T\asymp\sA\sqrt{L/L_2}$,
$K\asymp\mathcal R$, and
$\eta\mathcal R=\sqrt{LL_2/\sA}$.  The Bell estimate
$B_J\le J^J$ gives the second row.  Since
$\mathcal L_J(x,y)\ll L_2+J$ and
$\log(1+L/T)=O(L_2)$, the next two rows follow.  The two largest purely
combinatorial orders are at most $L_2^2/\sA$ up to a factor
$O(\log_4 x/\sA)$, and
\[
 \frac{L_2^2/\sA}{\sqrt{LL_2/\sA}}
 =\frac{L_2^{3/2}}{\sqrt{\sA L}}\longrightarrow0.
\]
Formula \eqref{eq:formal-encoding-factor} gives the encoding row.
The weighted-Euler-product row is
\eqref{eq:moving-tail-error}, with admissibility supplied by
Lemma~\ref{lem:moving-admissibility}.  Finally,
$JK=o(\eta L)$ is \eqref{eq:moving-overheads}.  Summing the normalized
bounds proves the proposition.
\end{proof}

\begin{lemma}[Named moving error functions]
\label{lem:named-moving-errors}
Put
\[
 H_{\rm src}=2^{-J}x^{1/2}Z^{-2J},\qquad
 v_{\rm src}=\frac{\log H_{\rm src}}K,
\]
and put
\[
 u_{\rm tr}^{-}=\frac{L-3K+\log 2}{2K},
 \qquad
 u_{\rm tr}^{+}=\frac{L-K+\log 2}{K}.
\]
For a sufficiently large absolute constant $C$, define
\begin{equation}\label{eq:epsilon-tr}
 \epsilon_{\rm tr}(x)
 =\frac{C}{\eta\mathcal R}
 \left(L_2+\log_2(3Z^2)
       +\frac{u_{\rm tr}^{+}}{\log u_{\rm tr}^{+}}\right).
\end{equation}
Then $\epsilon_{\rm tr}(x)\to0$.  Moreover
\begin{equation}\label{eq:named-moving-margins}
 v_{\rm src}\log v_{\rm src}\ge K+3\eta\mathcal R,
 \qquad
 u_{\rm tr}^{-}\log u_{\rm tr}^{-}\ge K+3\eta\mathcal R
\end{equation}
for all sufficiently large $x$.
\end{lemma}

\begin{proof}
One has
\[
 \log H_{\rm src}
 =\frac12L-2JK-J\log 2
 =\frac12L-o(\eta L)
\]
by Proposition~\ref{lem:moving-dependence-ledger}.  Also
\[
 \frac12(L-3K+\log 2)=\frac12L-o(\eta L).
\]
Apply the final part of Lemma~\ref{lem:moving-bookkeeping} to these two
lower thresholds.  It gives \eqref{eq:named-moving-margins}.  In addition,
$u_{\rm tr}^{+}\ll\sqrt{L/L_2}$, so
\[
 \frac{u_{\rm tr}^{+}/\log u_{\rm tr}^{+}}
      {\eta\mathcal R}\longrightarrow0.
\]
The remaining terms in \eqref{eq:epsilon-tr}
tend to zero by \eqref{eq:moving-tail-error}, while
$L_2/(\eta\mathcal R)\to0$ directly and
$\epsilon_{\rm dep}(x)\to0$ by
Proposition~\ref{lem:moving-dependence-ledger}.
\end{proof}

\begin{corollary}[The tail at the moving parameters]
\label{cor:moving-weight-tail}
Let $1\le j\le J$, let $R=Z$, and suppose $3\le H\le2x$ and
$\log H\ge\frac12L-\delta_H(x)\eta L$ for some
$\delta_H(x)\to0$.  Put $v=\log H/K$ and
\[
 \epsilon_{\rm wt}(H,j;x)=
 \frac{Cj}{\eta\mathcal R}
 \left(\log_2(3Z)+\frac{v}{\log v}\right),
\]
where $C$ is the constant in Lemma~\ref{lem:growing-weight-tail}.  Then
$\epsilon_{\rm wt}(H,j;x)\to0$ uniformly for $1\le j\le J$, and
\begin{equation}\label{eq:moving-weight-tail-explicit}
 \sum_{\substack{m>H\\P^+(m)\le Z}}\frac{\tau_j(m)}m
 \le\exp\{-v\log v+\epsilon_{\rm wt}(H,j;x)\eta\mathcal R\}.
\end{equation}
In particular,
\begin{equation}\label{eq:moving-weight-tail}
 \sum_{\substack{m>H\\P^+(m)\le Z}}\frac{\tau_j(m)}m
 \le \exp\{-K-2\eta\mathcal R\}
\end{equation}
for all sufficiently large $x$.
\end{corollary}

\begin{proof}
The explicit estimate is Lemma~\ref{lem:growing-weight-tail}.  Uniform
convergence of $\epsilon_{\rm wt}$ follows from
\eqref{eq:moving-tail-error}.  Lemma~\ref{lem:moving-bookkeeping} gives
$v\log v\ge K+3\eta\mathcal R$, and the final assertion follows.
\end{proof}

The next estimate replaces Proposition~\ref{prop:b4-general} when the
outer cutoff is $Z=e^K$, well beyond the original saddle.

\begin{proposition}[The large-prime transfer at moving rank]
\label{prop:moving-b4}
With the moving parameters above,
\begin{equation}\label{eq:moving-b4}
 \sum_{Z<p\le(2xZ)^{1/2}}
 \Psi_\tau\left(\frac{2xZ}{p^2},p\right)
 \ll x\exp\{-K+\epsilon_{\rm tr}(x)\eta\mathcal R\}.
\end{equation}
\end{proposition}

\begin{proof}
Put $X_p=2xZ/p^2$ and split at $p=Z^2$.  For $p>Z^2$, the elementary
bound $\Psi_\tau(X,p)\ll X\log(2X+2)$ gives
\begin{align*}
 \sum_{p>Z^2}\Psi_\tau(X_p,p)
 &\ll xZL\sum_{p>Z^2}\frac1{p^2}\\
 &\ll x\exp\{-K+L_2+O(1)\}\\
 &\ll x\exp\{-K+\epsilon_{\rm tr}(x)\eta\mathcal R\}.
\end{align*}

Now let $Z<p\le Z^2$ and put $u_p=\log X_p/\log p$.  Uniformly in this
range,
\[
 u_p\ge\frac{L-3K+O(1)}{2K}.
\]
The minimum is $u_{\rm tr}^{-}$ at the endpoint $p=Z^2$.
Lemma~\ref{lem:named-moving-errors} therefore gives
\begin{equation}\label{eq:up-moving-margin}
 u_p\log u_p\ge u_{\rm tr}^{-}\log u_{\rm tr}^{-}
 \ge K+3\eta\mathcal R.
\end{equation}
Also $u_p\le p^{1/2}$ and $u_p\to\infty$.  Applying the general form
\eqref{eq:weighted-tail-general} of Lemma~\ref{lem:weighted-smooth}, and
using the uniform bound encoded in \eqref{eq:epsilon-tr}, we obtain
\[
 \Psi_\tau(X_p,p)
 \le X_p\exp\{-K-2\eta\mathcal R\}.
\]
Since $\sum_{p>Z}p^{-2}\ll Z^{-1}$,
\[
 \sum_{Z<p\le Z^2}\Psi_\tau(X_p,p)
 \ll xZ e^{-K-2\eta\mathcal R}\sum_{p>Z}\frac1{p^2}
 \ll xe^{-K-2\eta\mathcal R}.
\]
This proves the proposition.
\end{proof}

\begin{lemma}[Quantitative encoding of large low-rank source products]
\label{lem:moving-low-encoding}
Let the moving parameters be as in
\eqref{eq:moving-parameters}--\eqref{eq:moving-JK}.  Let \(\mathcal D\)
be any family of products
\[
 d=\prod_{\nu=1}^{s}q_\nu^{a_\nu},\qquad s<J,
\]
of distinct local prime powers with \(q_\nu^{a_\nu}\le 2x\), such that
\[
 P^+\!\left(\varphi(q_\nu^{a_\nu})\right)\le Z,
 \qquad
 \Omega_y\!\left(\prod_{\nu=1}^{s}
              \varphi(q_\nu^{a_\nu})\right)<J,
\]
and every local factor has a prime divisor of its totient exceeding \(y\).
Then, for every \(H\ge3\),
\begin{equation}\label{eq:moving-low-encoding}
 \sum_{\substack{d\in\mathcal D\\ d>H}}\frac1d
 \le
 \exp\{O(J/y)\}
 \sum_{\substack{m>2^{-J}HZ^{-2J}\\P^+(m)\le Z}}
       \frac{\tau_J(m)}m.
\end{equation}
The implied constant is absolute.
\end{lemma}

\begin{proof}
Separate the factors with exponent at least two from the linear factors.
If \(a_\nu\ge2\), then necessarily \(q_\nu>y\): otherwise neither
\(q_\nu\) nor a prime divisor of \(q_\nu-1\) can exceed \(y\).
Moreover \(q_\nu\le Z\), and the occurrence of
\(q_\nu^{a_\nu-1}\) in the totient shows that
\[
 \sum_{a_\nu\ge2}(a_\nu-1)<J.
\]
Consequently the product \(d_0\) of all higher prime-power factors
satisfies
\begin{equation}\label{eq:low-power-product}
 d_0\le Z^{2J}.
\end{equation}
Their total reciprocal mass, summed over every possible collection, is at
most
\[
 \left(1+\sum_{q>y}\sum_{a\ge2}q^{-a}\right)^J
 \le \exp\{O(J/y)\}.
\]

Fix such a collection.  Write the remaining linear factors as distinct
primes \(q_1,\ldots,q_t\), with \(t<J\).  Since
\(P^+(q_i-1)\le Z\), the integers \(r_i=q_i-1\) are \(Z\)-smooth.  If
\(d=d_0q_1\cdots q_t>H\), then, by \eqref{eq:low-power-product},
\[
 r_1\cdots r_t
 >2^{-t}\frac{H}{d_0}
 \ge 2^{-J}HZ^{-2J}.
\]
Also \(q_i^{-1}\le r_i^{-1}\).  Order the \(r_i\)'s increasingly and
pad the resulting tuple with ones to length \(J\).  For a fixed product
\(m\), the number of such padded tuples is no larger than
\(\tau_J(m)\).  Summing first over the linear factors and then over the
higher factors proves \eqref{eq:moving-low-encoding}.
\end{proof}

\begin{corollary}[Moving supplier-system tail]
\label{cor:moving-high-encoding}
With the moving parameters and $H\ge3$,
\begin{equation}\label{eq:moving-high-encoding}
 \sum_{\mathcal P}\sum_{d>H}\frac{w_{\mathcal P}(d)}d
 \le \exp\{\epsilon_{\rm dep}(x)\eta\mathcal R\}
 \sum_{\substack{m>2^{-J}HZ^{-2J}\\P^+(m)\le Z}}
       \frac{\tau_J(m)}m.
\end{equation}
The displayed function $\epsilon_{\rm dep}(x)$ is independent of the
source systems and tends to zero.
\end{corollary}

\begin{proof}
Apply Proposition~\ref{prop:supplier-aggregate-tail} with $Y=Z$.
The bound for its encoding factor is
$\log\mathcal E_J(x;y,T)\le
\epsilon_{\rm dep}(x)\eta\mathcal R$ by
Proposition~\ref{lem:moving-dependence-ledger}.
\end{proof}

\begin{proposition}[Moving low-rank compression]
\label{prop:moving-low-rank}
Uniformly for $x\le X\le2x$,
\begin{multline}\label{eq:moving-low-rank}
 \#\{N\le X:P^+(\varphi(N))\le Z,
              \ \Omega_y(\varphi(N))<J\}\\
 \ll X\exp\{-K-c\eta\mathcal R\}
\end{multline}
for some absolute $c>0$ and all sufficiently large $x$.
\end{proposition}

\begin{proof}
For $N=\prod q^a$, let $d$ be the product of all full prime powers
$q^a\Vert N$ for which $\varphi(q^a)$ has a prime factor exceeding $y$.
There are fewer than $J$ such local factors, and, on writing $N=dm$,
\[
 P^+(\varphi(m))\le y.
\]

We first bound the reciprocal mass of the possible $d$.  For a linear
factor $q$, let $\mathcal B$ be the complete multiset of prime factors of
$q-1$ lying in $(y,Z]$.  Its cardinality is less than $J$.  If
$D_B=D(\mathcal B)$, Brun--Titchmarsh and partial summation give, uniformly
in $D_B$,
\[
 \sum_{\substack{q\le2x\\q\ \mathrm{prime}\\D_B\mid q-1}}\frac1q
 \ll \frac{L_2}{\varphi(D_B)}
 \ll \frac{L_2}{D_B},
\]
because $D_B/\varphi(D_B)\le\exp(O(J/y))\le2$.  Moreover,
\[
 S:=\sum_{y<p\le Z}\frac1p\ll 1+\log\frac{K}{T}
 \ll 1+\log(2J).
\]
Therefore the reciprocal mass of all possible linear local factors is at
most
\begin{equation}\label{eq:moving-one-linear-mass}
 Q\ll L_2\sum_{0<r<J}S^r
 \le L_2(1+S)^J.
\end{equation}
For a higher prime power, $q>y$, $q\le Z$, and its exponent satisfies
$a\le J$, because $q^{a-1}\mid\varphi(q^a)$.  Hence the reciprocal mass
of all higher prime-power local factors is $O(1)$.  It follows that
\begin{equation}\label{eq:moving-d-mass}
 \sum_d\frac1d\le(1+Q+O(1))^J
 \le\exp\{C\epsilon_{\rm dep}(x)\eta\mathcal R\}.
\end{equation}

If $d\le x^{1/2}$, then $X/d\ge x^{1/2}$.  Lemma~\ref{lem:bfps-uniform}\textup{(ii)} gives, uniformly for
$x^{1/2}\le W\le2x$,
\begin{equation}\label{eq:bfps-half-range}
 \Phi(W,y)
 \le W\exp\{-E+\epsilon_{\rm BFPS}(x)\eta\mathcal R\}.
\end{equation}  Summing \eqref{eq:bfps-half-range} with
\eqref{eq:moving-d-mass} gives exponent
\[
 -E+\{\epsilon_{\rm BFPS}(x)+C\epsilon_{\rm dep}(x)\}
       \eta\mathcal R.
\]
Since $E-K\ge\frac12\eta\mathcal R$ for large $x$ and both named
error functions tend to zero, this is at most
$-K-\tfrac14\eta\mathcal R$ for all sufficiently large $x$.

It remains to treat $d>x^{1/2}$.  Apply
Lemma~\ref{lem:moving-low-encoding} with $H=x^{1/2}$.  It gives
\[
 \sum_{d>x^{1/2}}\frac1d
 \le e^{O(J/y)}
 \sum_{\substack{m>H_0\\P^+(m)\le Z}}
       \frac{\tau_J(m)}m,
 \qquad
 H_0=2^{-J}x^{1/2}Z^{-2J}.
\]
By Lemma~\ref{lem:moving-bookkeeping},
$\log H_0=\frac12L-o(\eta L)$, so Corollary
\ref{cor:moving-weight-tail} bounds the last expression by
$e^{-K-2\eta\mathcal R}$.  Since $J/y=o(1)$, counting multiples of $d$
completes the proof.
\end{proof}

The high-rank argument has the same main term as in fixed rank, but the
large-supplier tail must again be aggregated with $\tau_J$.

\begin{proposition}[Moving high-rank correlation]
\label{prop:moving-high-rank}
Uniformly for $1\le h\le y$,
\begin{multline}\label{eq:moving-high-rank}
 \#\{n\le x:\varphi(n)=\varphi(n+h)=M,\ P^+(M)\le Z,\\
                   \Omega_y(M)\ge J\}
 \ll x\exp\{-K+C\epsilon_{\rm dep}(x)\eta\mathcal R\}
\end{multline}
for an absolute constant $C$.
\end{proposition}

\begin{proof}
Choose the canonical labelled tuple $\mathcal P$ from the first $J$
large-prime occurrences of the common value.  Lemma~\ref{lem:supplier-coverage}
provides the weight on both shifted integers.  From
Lemma~\ref{lem:multiset-supplier-mass},
\begin{align}
 \Lambda(\mathcal P)
 &\le\frac{\exp\{C\epsilon_{\rm dep}(x)\eta\mathcal R\}}
              {D(\mathcal P)},
 \label{eq:moving-Lambda}\\
 \|F_{\mathcal P}\|_{\infty,2x}
 &\le\exp\{C\epsilon_{\rm dep}(x)\eta\mathcal R\}.
 \label{eq:moving-F-infty}
\end{align}
Indeed the logarithms of the Bell factor, the block masses, and the
pointwise factor are entries in the dependence ledger.

All primes in the support of the divisor weight exceed $y\ge h$.
Lemma~\ref{lem:shifted-divisor-convolution} therefore gives, for each
$\mathcal P$,
\[
 \sum_{n\le x}F_{\mathcal P}(n)F_{\mathcal P}(n+h)
 \le2x\Lambda(\mathcal P)^2
 +4x\|F_{\mathcal P}\|_{\infty,2x}
       \Lambda_{>\sqrt x}(\mathcal P).
\]
The main terms satisfy
\begin{align}
 2x\sum_{\mathcal P}\Lambda(\mathcal P)^2
 &\le x\exp\{C\epsilon_{\rm dep}(x)\eta\mathcal R\}
   \left(\sum_{p>y}p^{-2}\right)^J\notag\\
 &\le x\exp\{-JT+C\epsilon_{\rm dep}(x)\eta\mathcal R\}\notag\\
 &=x\exp\{-K+C\epsilon_{\rm dep}(x)\eta\mathcal R\}.
 \label{eq:moving-high-main}
\end{align}

For the large-divisor terms, Corollary~\ref{cor:moving-high-encoding}
with $H=\sqrt x$ gives
\[
 \sum_{\mathcal P}\Lambda_{>\sqrt x}(\mathcal P)
 \le e^{\epsilon_{\rm dep}(x)\eta\mathcal R}
 \sum_{\substack{m>H_{\rm src}\\P^+(m)\le Z}}
       \frac{\tau_J(m)}m.
\]
By Lemma~\ref{lem:named-moving-errors} and
Lemma~\ref{lem:growing-weight-tail}, the last sum is at most
$e^{-K-2\eta\mathcal R}$ for all sufficiently large $x$.  Multiplying by
\eqref{eq:moving-F-infty} leaves
\[
 x\exp\{-K-\eta\mathcal R\},
\]
which is smaller than \eqref{eq:moving-high-main}.  This proves the
proposition.
\end{proof}

\begin{theorem}[Moving-rank diagonal decomposition]
\label{thm:moving-rank-decomposition}
With the parameters \eqref{eq:moving-parameters}--\eqref{eq:moving-JK},
there is a function $\varepsilon_{\rm mov}(x)\to0$ such that, uniformly
for positive integers $h\le y$,
\begin{equation}\label{eq:moving-rank-decomposition}
 S_h^\varphi(x)=D_{h,>Z}^\varphi(x)
 +O\left(x\exp\{-K+\varepsilon_{\rm mov}(x)\eta\mathcal R\}\right).
\end{equation}
Consequently,
\begin{equation}\label{eq:moving-rank-simplified}
 S_h^\varphi(x)=D_{h,>Z}^\varphi(x)
 +O\left(x\exp\{-(\tfrac12-o(1))\sqrt{\log x\log_2 x}\}\right).
\end{equation}
If $h$ is odd, the diagonal is empty.
\end{theorem}

\begin{proof}
Let $M=\varphi(n)=\varphi(n+h)$.  If $P^+(M)>Z$, the set-theoretic identity
\eqref{eq:diagonal-cardinality-identity}, with $y=Z$, gives
\[
 0\le
 \#\{n\le x:P^+(M)>Z\}-D_{h,>Z}^\varphi(x)
 \le \mathcal E_h(x;Z).
\]
The square-source, small-cofactor, and reciprocal terms in
Proposition~\ref{prop:transfer-diagonal} are
\[
 \ll \frac{x(1+L+K^2)}{e^K}+x^{1/2}
 \ll x\exp\{-K+\epsilon_{\rm tr}(x)\eta\mathcal R\},
\]
and Proposition~\ref{prop:moving-b4} treats the divisor-weighted term.
Therefore
\[
 0\le
 \#\{n\le x:P^+(M)>Z\}-D_{h,>Z}^\varphi(x)
 \ll xe^{-K+\epsilon_{\rm tr}(x)\eta\mathcal R}.
\]

Suppose $P^+(M)\le Z$.  If $\Omega_y(M)<J$, apply Proposition
\ref{prop:moving-low-rank}; if $\Omega_y(M)\ge J$, apply Proposition
\ref{prop:moving-high-rank}.  These cases are disjoint and exhaustive.
Define explicitly
\begin{equation}\label{eq:epsilon-mov-definition}
 \varepsilon_{\rm mov}(x)=
 C\bigl(\epsilon_{\rm BFPS}(x)+\epsilon_{\rm dep}(x)
        +\epsilon_{\rm tr}(x)\bigr),
\end{equation}
with $C$ large enough for the preceding three estimates.  Each summand
tends to zero by Lemmas~\ref{lem:bfps-uniform} and
\ref{lem:named-moving-errors}, and by
Proposition~\ref{lem:moving-dependence-ledger}; hence no unnamed diagonal choice is
present.  Finally \eqref{eq:moving-K} yields
\eqref{eq:moving-rank-simplified}, and Lemma~\ref{lem:J-finite} removes
the diagonal for odd $h$.
\end{proof}

\begin{proof}[Proof of Theorem~\ref{thm:scale-jump}]
Take the explicit parameters in
\eqref{eq:moving-parameters}--\eqref{eq:moving-JK} and apply
Theorem~\ref{thm:moving-rank-decomposition}.  The asymptotic formulae for
$T$ and $K$ are \eqref{eq:moving-asymptotics} and
\eqref{eq:moving-K}.
\end{proof}

\begin{proof}[Proof of Corollary~\ref{cor:scale-jump-intro}]
Use $h=1$ in Theorem~\ref{thm:scale-jump}; the diagonal is empty by
Lemma~\ref{lem:J-finite}.
\end{proof}

\begin{corollary}[Uniform odd shifts on the new scale]
\label{cor:scale-jump-h1}
Uniformly for odd $1\le h\le e^T$,
\begin{equation}\label{eq:scale-jump-h1}
 S_h^\varphi(x)
 \ll x\exp\left\{-(\tfrac12-o(1))
        \sqrt{\log x\log_2 x}\right\},
\end{equation}
where
\[
 T=(1+o(1))(\log_3 x+\log_4 x-\log 2)
       \sqrt{\frac{\log x}{\log_2 x}}.
\]
\end{corollary}

\begin{proof}
Apply Theorem~\ref{thm:moving-rank-decomposition} and use the parity
statement in Lemma~\ref{lem:J-finite}.
\end{proof}

\begin{proposition}[The moving-rank saddle of the present method]
\label{prop:moving-method-saddle}
Let $\mathcal R=(LL_2)^{1/2}$ and suppose the outer nonsmooth cutoff is
$Z=e^{c\mathcal R}$, where $c>0$ is fixed.  The transfer estimate obtained
from Proposition~\ref{prop:transfer-diagonal} and
Lemma~\ref{lem:weighted-smooth} is at most
\begin{equation}\label{eq:moving-method-saddle}
 x\exp\left\{-\left(\min\left(c,\frac1{4c}\right)+o(1)\right)
                    \mathcal R\right\}.
\end{equation}
Consequently the bound \eqref{eq:moving-method-saddle} is optimized at
$c=1/2$, where its exponent is $(1/2+o(1))\mathcal R$.
\end{proposition}

\begin{proof}
The square-source, small-cofactor, and reciprocal terms in
Proposition~\ref{prop:transfer-diagonal} have size
$xe^{-c\mathcal R+o(\mathcal R)}$.  In the intermediate range
$Z<p\le Z^2$, put $u_p=\log(2xZ/p^2)/\log p$.  At the worst endpoint,
\[
 u_p=\left(\frac1{2c}+o(1)\right)\sqrt{\frac{L}{L_2}},
 \qquad
 \log u_p=(\tfrac12+o(1))L_2.
\]
Thus
\[
 u_p\log u_p=\left(\frac1{4c}+o(1)\right)\mathcal R.
\]
The general divisor-weighted estimate
\eqref{eq:weighted-tail-general}, followed by summation of
$2xZ/p^2$, gives the second exponent in
\eqref{eq:moving-method-saddle}; the tail $p>Z^2$ contributes
$\ll x\log(2x)\sum_{p>Z^2}Z/p^2\ll xe^{-c\mathcal R+o(\mathcal R)}$,
which is absorbed into the first exponent.  Taking the smaller of the two
exponents proves the first assertion.  Finally
$\min(c,1/(4c))$ is maximized at $c=1/2$.
\end{proof}

\begin{remark}
Proposition~\ref{prop:moving-method-saddle} records the internal
optimization of the present transfer architecture.  Within these estimates,
the constant $1/2$ is the optimizing value.  It is the optimum of the
stated bounds rather than a limit of the method itself: the intermediate-range
estimate above is obtained by charging every prime the saving of the worst
endpoint $p=Z^2$, and is not tight, whereas the constraint $c\le1/2$ is
already forced by the source-height threshold of Lemma~\ref{lem:named-moving-errors},
whose exponent $\tfrac12$ is inherited from the cutoff $\Lambda_{>\sqrt x}$
of Lemma~\ref{lem:shifted-divisor-convolution}.
\end{remark}

\begin{proposition}[The classical $h=2$ family]\label{prop:h2-family}
Let $t\ge1$.  If $2t+1$ and $4t+1$ are both prime, then
\[
 n=2(4t+1)
\]
satisfies $\varphi(n)=\varphi(n+2)$.
\end{proposition}

\begin{proof}
One has $n+2=4(2t+1)$, and therefore
\[
 \varphi(n)=\varphi(2)\varphi(4t+1)=4t
 =\varphi(4)\varphi(2t+1)=\varphi(n+2).
\]
\end{proof}

\begin{remark}[Provenance]
This illustrative $h=2$ construction is classical and goes back to
Moser \cite{Moser}; Schinzel extended the construction to general even
shifts \cite{Schinzel}.  It is included as a concrete classical member of the even-shift diagonal.
\end{remark}

\begin{remark}
Theorem \ref{thm:all-shifts} is the amplified all-shifts form of the
result.  For odd shifts the finite diagonal is empty.  For even shifts it
contains the Graham--Holt--Pomerance prime-pair mechanism \cite{GHP}, while everything
outside that specified diagonal enjoys every fixed amplified rank.
\end{remark}

\section{Further consequences and variants}

\subsection{Immediate consequences for chains}
For integers $r\ge2$ and $h\ge1$, define
\[
 \mathcal C_{r,h}^{\varphi}(x)
 =\#\{n\le x:\varphi(n)=\varphi(n+h)=\cdots
                   =\varphi(n+(r-1)h)\}.
\]
For $Y\ge2$, let $\mathcal D_{r,h,>Y}^{\varphi}(x)$ count the integers in
this chain set whose first pair belongs to $D_{h,>Y}^{\varphi}(x)$.  Both
quantities count integers, and
\begin{equation}\label{eq:chain-set-inclusions}
 0\le\mathcal D_{r,h,>Y}^{\varphi}(x)
 \le D_{h,>Y}^{\varphi}(x),
 \qquad
 \mathcal C_{r,h}^{\varphi}(x)\le S_h^{\varphi}(x).
\end{equation}

\begin{corollary}[Pair bounds inherited by chains]
\label{cor:chain-inherited}
Fix $r\ge2$ and a fixed integer $Q\ge1$.  Put
\[
 T_Q=\frac{G}{\sqrt Q},\qquad Y_Q=e^{\sqrt QG}.
\]
For $x\ge x_0(Q)$, uniformly for $1\le h\le e^{T_Q}$,
\begin{equation}\label{eq:chain-inherited-fixed}
 \mathcal C_{r,h}^{\varphi}(x)
 =\mathcal D_{r,h,>Y_Q}^{\varphi}(x)
 +O_Q\!\left(xe^{-\sqrt QG+\varepsilon_Q(x)V}\right).
\end{equation}
With the moving parameters of Theorem~\ref{thm:scale-jump}, uniformly for
$1\le h\le y$,
\begin{equation}\label{eq:chain-inherited-moving}
 \mathcal C_{r,h}^{\varphi}(x)
 =\mathcal D_{r,h,>Z}^{\varphi}(x)
 +O\!\left(xe^{-K+\varepsilon_{\rm mov}(x)\eta\mathcal R}\right).
\end{equation}
For odd $h$ both diagonal terms vanish.
\end{corollary}

\begin{proof}
Let $\mathscr C$ be the set of chain solutions and let $\mathscr D_Y$ be
the first-pair diagonal set.  Then
\[
 0\le |\mathscr C|-|\mathscr C\cap\mathscr D_Y|
 \le |\{n\le x:\varphi(n)=\varphi(n+h)\}|-|\mathscr D_Y|.
\]
The right side is exactly the off-diagonal cardinality bounded by
Theorem~\ref{thm:main} at rank $Q$, or by
Theorem~\ref{thm:scale-jump} in the moving case.  This proves both
formulas.  Odd shifts have empty pair diagonal by
Lemma~\ref{lem:J-finite}.
\end{proof}

\begin{remark}[No separate chain amplification]
The proof uses only the first equality
$\varphi(n)=\varphi(n+h)$.  No additional saving is obtained from the
remaining $r-2$ equalities.  For example, choosing $Q=2$ in
\eqref{eq:chain-inherited-fixed} gives
\[
 \#\{n\le x:\varphi(n)=\varphi(n+1)=\varphi(n+2)\}
 \ll xe^{-\sqrt2G+\varepsilon_2(x)V},
\]
but this is merely the rank-two pair bound restricted to a smaller set.
\end{remark}

\subsection{Rank-one and optimization consequences}

\begin{corollary}[Rank-one nonsmooth estimate]\label{cor:nonsmooth-saddle}
Let $y=e^G$.  Uniformly for positive odd $h\le y$,
\[
 B_h(x;y)\ll xe^{-G+o(V)}.
\]
\end{corollary}

\begin{proof}
Combine Theorem~\ref{thm:transfer}, Corollary~\ref{cor:b4}, and
Lemma~\ref{lem:scales}.
\end{proof}

\begin{remark}[Optimization statements]
The variable-cutoff and moving-cutoff calculations in
Propositions~\ref{prop:variable-saddle} and
\ref{prop:moving-method-saddle} describe the best exponents furnished by
the decompositions and estimates used in this paper.  They identify the
internal saddle points of the method.
\end{remark}

\section*{Acknowledgements}
The author acknowledges the use of OpenAI's ChatGPT during the preparation of this manuscript.  While it was used for ideation, formulation, proof exploration and refinement, narrowing the search space, programming, LaTeX formatting and other forms of orchestration, the author nonetheless takes full responsibility for the accuracy of the final contents of this paper.

The accompanying Lean formalisation~\cite{LeanFormalization} was developed by the author with Codex and Claude Code; the author likewise takes full responsibility for the formal statements' fidelity to the results of this paper.

\end{document}